\documentclass[reqno, twoside, 11pt, a4paper]{article}
\usepackage{amsfonts, amsmath, amssymb, amsthm, color}
\usepackage[hmargin=2.5cm, vmargin=2.5cm]{geometry}
\usepackage{txfonts}
\usepackage{graphicx, setspace}
\usepackage{hyperref}

\newtheorem{thm}{Theorem}[section]
\newtheorem{lemma}[thm]{Lemma}
\newtheorem{prop}[thm]{Proposition}
\newtheorem{cor}[thm]{Corollary}

\newtheorem{thmx}{Theorem}

\newtheorem{propx}[thmx]{Proposition}

\theoremstyle{definition}
\newtheorem{rmk}[thm]{Remark}
\newtheorem{conj}[thm]{Conjecture}

\def\bysame{\leavevmode\hbox to3em{\hrulefill}\thinspace}

\newcommand{\ga}{\gamma}
\newcommand{\ep}{\epsilon}
\newcommand{\ve}{\varepsilon}
\newcommand{\vp}{\varphi}
\newcommand{\vr}{\varrho}
\newcommand{\vt}{\vartheta}
\newcommand{\pa}{\partial}

\newcommand{\ls}{{\lambda,\sigma}}

\newcommand{\bg}{\bar{g}}
\newcommand{\bx}{\bar{x}}
\newcommand{\by}{\bar{y}}
\newcommand{\hh}{\hat{h}}
\newcommand{\hw}{\hat{w}}
\newcommand{\ov}{\overline{V}}
\newcommand{\ox}{\overline{X}}
\newcommand{\tal}{\tilde{\alpha}}
\newcommand{\trho}{\tilde{\rho}}
\newcommand{\wtu}{\widetilde{U}}
\newcommand{\wtv}{\widetilde{V}}

\newcommand{\wPhi}{\widetilde{\Phi}}
\newcommand{\wPsi}{\widetilde{\Psi}}
\newcommand{\whw}{\widehat{W}}

\newcommand{\xr}{(X;\rho^{1-2\ga})}

\newcommand{\mb}{\mathbb{B}}
\newcommand{\mn}{\mathbb{N}}
\newcommand{\mr}{\mathbb{R}}
\newcommand{\ms}{\mathbb{S}}

\newcommand{\mcc}{\mathcal{C}}
\newcommand{\mcd}{\mathcal{D}}
\newcommand{\mcf}{\mathcal{F}}
\newcommand{\mcn}{\mathcal{N}}
\newcommand{\mcr}{\mathcal{R}}
\newcommand{\mcw}{\mathcal{W}}

\newcommand{\la}{\left\langle}
\newcommand{\ra}{\right\rangle}

\renewcommand{\(}{\left(}
\renewcommand{\)}{\right)}

\begin{document}
\title{Existence theorems of the fractional Yamabe problem}
\author{Seunghyeok Kim, Monica Musso, Juncheng Wei}

\newcommand{\Addresses}{{\bigskip \footnotesize

\medskip
\noindent (Seunghyeok Kim) \textsc{Facultad de Matem\'{a}ticas, Pontificia Universidad Cat\'{o}lica de Chile, Avenida Vicu\~{n}a Mackenna 4860, Santiago, Chile}\par\nopagebreak
\noindent \textit{E-mail address}: \texttt{shkim0401@gmail.com}

\medskip
\noindent (Monica Musso) \textsc{Facultad de Matem\'{a}ticas, Pontificia Universidad Cat\'{o}lica de Chile, Avenida Vicu\~{n}a Mackenna 4860, Santiago, Chile}\par\nopagebreak
\noindent \textit{E-mail address}: \texttt{mmusso@mat.puc.cl}

\medskip
\noindent (Juncheng Wei) \textsc{Department of Mathematics, University of British Columbia, Vancouver, B.C., Canada, V6T 1Z2}\par\nopagebreak
\noindent \textit{E-mail address}: \texttt{jcwei@math.ubc.ca}}}

\date{\today}
\maketitle

\begin{abstract}
Let $X$ be an asymptotically hyperbolic manifold and $M$ its conformal infinity.
This paper is devoted to deduce several existence results of the fractional Yamabe problem on $M$ under various geometric assumptions on $X$ and $M$:
Firstly, we handle when the boundary $M$ has a point at which the mean curvature is negative.
Secondly, we re-encounter the case when $M$ has zero mean curvature and is either non-umbilic or umbilic but non-locally conformally flat.
As a result, we replace the geometric restrictions given by Gonz\'alez-Qing (2013) \cite{GQ} and Gonz\'alez-Wang (2015) \cite{GW} with simpler ones.
Also, inspired by Marques (2007) \cite{Ma2} and Almaraz (2010) \cite{Al}, we study lower-dimensional manifolds.
Finally, the situation when $X$ is Poincar\'e-Einstein, $M$ is either locally conformally flat or 2-dimensional is covered
under the validity of the positive mass theorem for the fractional conformal Laplacians.
\end{abstract}



\allowdisplaybreaks
\numberwithin{equation}{section}

\section{Introduction and the Main Results}
Given $n \in \mn$, let $X^{n+1}$ be an $(n+1)$-dimensional smooth manifold with smooth boundary $M^n$.
A function $\rho$ in $X$ is called a {\it defining function} of the boundary $M$ in $X$ if
$$
\rho >0 \quad \text{in } X, \quad \rho = 0 \quad \text{on } M \quad \text{and} \quad d\rho \not= 0 \quad \text{on } M.
$$
A metric $g^+$ in $X$ is {\it conformally compact} with {\it conformal infinity} $(M, [\hh])$ if there exists a boundary defining function $\rho$
so that the closure $(\ox, \bg)$ of $X$ is compact for $\bg := \rho^2 g^+$ and $\bg|_{M} \in [\hh]$.
A manifold $(X^{n+1},g^+)$ is said to be {\it asymptotically hyperbolic} if $g^+$ is conformally compact and $|d\rho|_{\bg} \to 1$ as $\rho \to 0$.
Also if $(X, g^+)$ is asymptotically hyperbolic and Einstein, then it is called {\it Poincar\'e-Einstein} or {\it conformally compact Einstein}.

Suppose that an asymptotically hyperbolic manifold $(X, g^+)$ with the conformal infinity $(M^n, [\hh])$ is given.
Also, for any $\ga \in (0,1)$, let $P^{\ga}_{\hh} = P^{\ga}[g^+,\hh]$ be the {\it fractional conformal Laplacian} whose principle symbol is equal to $(-\Delta_{\hh})^{\ga}$ (see \cite{GZ} for its precise definition).
In this article, we are interested in finding a conformal metric $\hh$ on $M$ with constant {\it fractional scalar curvature} $Q^{\ga}_{\hh} = P^{\ga}_{\hh}(1)$.
This problem is referred to be the {\it fractional Yamabe problem} or the {\it $\ga$-Yamabe problem}, and it was introduced and investigated by Gonz\'alez-Qing \cite{GQ} and Gonz\'alez-Wang \cite{GW}.
By imposing some restrictions on the dimension and geometric behavior of the manifold, the authors obtained the existence results when $M$ is non-umbilic or it is umbilic but not locally conformally flat.
Here we relieve the hypotheses made in \cite{GQ, GW} and examine when the bubble (see \eqref{eq_bubble} below for its precise definition) cannot be used as an appropriate test function.

\medskip
As its name alludes, the fractional conformal Laplacian $P^{\ga}_{\hh}$ has the {\it conformal covariance property}: It holds that
\begin{equation}\label{eq_conf_cov}
P^{\ga}_{\hh_w}(u) = w^{-{n+2\ga \over n-2\ga}} P^{\ga}_{\hh}(wu)
\end{equation}
for a conformal change of the metric $\hh_w = w^{4/(n-2\ga)} \hh$.
Hence the fractional Yamabe problem can be formulated as looking for a positive solution of the nonlocal equation
\begin{equation}\label{eq_yamabe}
P^{\ga}_{\hh} u = c u^{n+2\ga \over n-2\ga} \quad \text{on } M
\end{equation}
for some $c \in \mr$ provided $n > 2\ga$.
On the other hand, if $\ga = 1$, $P^{\ga}_{\hh}$ and $Q^{\ga}_{\hh}$ precisely match with the classical {\it conformal Laplacian} $L_{\hh}$ and a constant multiple of the scalar curvature $R[\hh]$ on $(M,\hh)$
\begin{equation}\label{eq_conf_L}
P^1_{\hh} = L_{\hh} := - \Delta_{\hh} + {n-2 \over 4(n-1)} R[\hh] \quad \text{and} \quad Q^1_{\hh} = {n-2 \over 4(n-1)} R[\hh],
\end{equation}
respectively. If $\ga = 2$, they coincide with the Paneitz operator \cite{Pa} and Branson's $Q$-curvature \cite{B}.
Hence the $1$ or $2$-Yamabe problems are reduced to the classical Yamabe problem and the $Q$-curvature problem.

Thanks to the efforts of various mathematicians, a complete solution of the Yamabe problem has been known.
After Yamabe \cite{Ya} raised the problem and suggested an outline of the proof,
Trudinger \cite{Tr} first obtained a least energy solution to \eqref{eq_yamabe} under the setting that the scalar curvature of $(M, \hh)$ is nonpositive.
Successively, Aubin \cite{Au} examined the case when $n \ge 6$ and $M$ is non-locally conformally flat,
and Schoen \cite{Sc} gave an affirmative answer when $n = 3, 4, 5$ or $M$ is locally conformally flat by using the positive mass theorem \cite{SY1, SY2, SY4}.
In Lee-Parker \cite{LP}, the authors provided a new proof which unified the local proof of Aubin and the global proof of Scheon, introducing the notion of the conformal normal coordinates.

Also there have been lots of results on the $Q$-curvature problem $(\ga = 2)$ for 4-dimensional manifolds $(M^4,[\hh])$.
By the Chern-Gauss-Bonnet formula, the {\it total $Q$-curvature}
\[k_P := \int_{M^4} Q^2_{\hh}\, dv_{\hh}\]
is a conformal invariant.
Gursky \cite{Gu} proved that if a manifold $M^4$ has the positive Yamabe constant $\Lambda^1(M,[\hh]) > 0$ (see \eqref{eq_yamabe_f}) and satisfies $k_P \ge 0$,
then its Paneitz operator $P^2_{\hh}$ has the properties
\begin{equation}\label{eq_Q_con}
\text{ker}\, P^2_{\hh} = \mr \quad \text{and} \quad P^2_{\hh} \ge 0.
\end{equation}
Also Chang-Yang \cite{CY} proved that any compact $4$-manifold such that \eqref{eq_Q_con} and $k_P < 8 \pi^2$ hold has a solution to
\[P^2_{\hh} u + 2Q^2_{\hh} u = 2c e^{4u} \quad \text{on } M, \quad c \in \mr\]
where $Q^2_{\hh}$ is the $Q$-curvature.
This result was generalized by Djadli-Malchiodi \cite{DM} where only $\text{ker}\, P^2_{\hh} = \mr$ and $k_P \ne 8m \pi^2$ for all $m \in \mn$ are demanded.
For other dimensions than 4, Gursky-Malchiodi \cite{GM} recently discovered the strong maximum principle of $P^2_{\hh}$ for manifolds $M^n$ ($n \ge 5$) with non-negative scalar curvature and semi-positive $Q$-curvature.
Motivated by this result, Hang and Yang developed the existence theory of \eqref{eq_yamabe} for a general class of manifolds $M^n$ including ones
such that $\Lambda^1(M,[\hh]) > 0$ and there exists $\hh' \in [\hh]$ with $Q^2_{\hh'} > 0$, provided $n \ge 5$ \cite{HY2, HY4} or $n = 3$ \cite{HY, HY2, HY3}.
In \cite{HY4}, the positive mass theorem for the Paneitz operator \cite{HR, GM} was used to construct a test function.
We also point out that a solution to \eqref{eq_yamabe} was obtained in \cite{QR} for a locally conformally flat manifold $(n \ge 5)$ with positive Yamabe constant and Poincar\'e exponent less than $(n-4)/2$.

In addition, when $\ga = 1/2$, the fractional Yamabe problem has a deep relationship with the boundary Yamabe problem proposed by Escobar \cite{Es}, who regarded it as a generalization of the Riemann mapping theorem:
It asks if a compact manifold $\ox$ with boundary is conformally equivalent to one of zero scalar curvature whose boundary $M$ has constant mean curvature.
It was solved by the series of works by Escobar himself \cite{Es, Es3}, Marques \cite{Ma, Ma2} and Almaraz \cite{Al}.
It is worthwhile to mention that there is another type of boundary Yamabe problem also suggested by Escobar \cite{Es2}:
Find a conformal metric such that the scalar curvature of $X$ is constant and the boundary $M$ is minimal.
It was further studied by Brendle-Chen \cite{BC}.

\medskip
In \cite{CG} (see also \cite{CC}), Chang and Gonz\'alez observed that the fractional conformal Laplacian,
defined through the scattering theory in Graham-Zworski \cite{GZ}, can be described in terms of Dirichlet-Neumann operators.
Especially, \eqref{eq_yamabe} has an equivalent extension problem, which is degenerate elliptic but local.
\begin{thmx}\label{thm_ext}
Suppose that $n > 2\ga$, $\ga \in (0,1)$, $(X, g^+)$ is an asymptotically hyperbolic manifold with conformal infinity $(M, [\hh])$.
Assume also that $\rho$ is a defining function associated to $M$ such that $|d \rho|_{\bg} = 1$ near $M$ (such $\rho$ is called geodesic),
and $\bg = \rho^2 g^+$ is a metric of the compact manifold $\ox$.
In addition, we let the mean curvature $H$ on $(M, \hh) \subset (\ox, \bg)$ be 0 if $\ga \in (1/2,1)$, and set
\begin{equation}\label{eq_E}
E(\rho) = \rho^{-1-s} (-\Delta_{g^+}-s(n-s))\rho^{n-s} \quad \text{in } X
\end{equation}
where $s := n/2 + \ga$. It can be shown that  \eqref{eq_E} is reduced to
\begin{equation} \label{eq_E1}
E(\rho)= \({n-2\ga \over 4n}\) \left[ R[\bg] - (n(n+1)+R[g^+]) \rho^{-2} \right] \rho^{1-2\ga} \quad \text{near } M
\end{equation}
where $R[\bg]$ and $R[g^+]$ are the scalar curvature of $(\ox, \bg)$ and $(X, g^+)$, respectively.

\medskip \noindent $(1)$ If a positive function $U$ satisfies
\begin{equation}\label{eq_yamabe_e_1}
\begin{cases}
-\textnormal{div}_{\bg}\(\rho^{1-2\ga}\nabla U\) + E(\rho)U = 0 &\text{in } (X, \bg),\\
U = u &\text{on } M
\end{cases} \end{equation}
and
\begin{equation}\label{eq_yamabe_e_2}
\pa^{\ga}_{\nu} U := - \kappa_{\ga} \(\lim_{\rho \to 0+} \rho^{1-2\ga}{\pa U \over \pa \rho}\) = \begin{cases}
cu^{n+2\ga \over n-2\ga} &\text{for } \ga \in (0,1) \setminus \{1/2\},\\
cu^{n+2\ga \over n-2\ga} - \(\dfrac{n-1}{2}\)Hu &\text{for } \ga = \{1/2\}
\end{cases} \end{equation}
on $M$, then $u$ solves \eqref{eq_yamabe}.
Here $\kappa_{\ga} > 0$ is the constant whose explicit value is given in \eqref{eq_const} below and $\nu$ stands for the outward unit normal vector with respect to the boundary $M$.

\medskip \noindent $(2)$ Assume further that the first $L^2$-eigenvalue $\lambda_1(-\Delta_{g^+})$ of the Laplace-Beltrami operator $-\Delta_{g^+}$ satisfies
\begin{equation}\label{eq_eig}
\lambda_1(-\Delta_{g^+}) > {(n-1)^2 \over 4} - \ga^2.
\end{equation}
Then there is a special defining function $\rho^*$ such that $E(\rho^*) = 0$ in $X$ and $\rho^*(\rho) = \rho\, (1 + O(\rho^{2\ga}))$ near $M$.
Furthermore the function $\wtu := (\rho/\rho^*)^{(n-2\ga)/2}U$ solves a degenerate elliptic equation of pure divergent form
\[\begin{cases}
-\textnormal{div}_{\bg^*}\(\(\rho^*\)^{1-2\ga}\nabla \wtu\) = 0 &\text{in } (X, \bg^*),\\
\pa^{\ga}_{\nu} \wtu = - \kappa_{\ga} \(\lim\limits_{\rho^* \to 0+} (\rho^*)^{1-2\ga}\dfrac{\pa \wtu}{\pa \rho^*}\) = P^{\ga}_{\hh} u - Q^{\ga}_{\hh} u = cu^{n+2\ga \over n-2\ga} - Q^{\ga}_{\hh} u &\text{on } M
\end{cases}\]
where $\bg^* := (\rho^*)^2 g^+$ and $Q^{\ga}_{\hh}$ is the fractional scalar curvature.
\end{thmx}

Notice that in order to seek a solution of \eqref{eq_yamabe}, it is natural to introduce the {\it $\ga$-Yamabe functional}
\begin{equation}\label{eq_yamabe_f}
I_{\hh}^{\ga}[u] = {\int_M u P^{\ga}_{\hh} u \ dv_{\hh} \over (\int_M u^{2n \over n-2\ga} dv_{\hh})^{n-2\ga \over n}} \quad \text{for } u \in C_c^{\infty}(M),\ u > 0 \text{ on } M
\end{equation}
and its infimum $\Lambda^{\ga}(M,[\hh])$, called the {\it $\ga$-Yamabe constant}.
By the previous theorem and the energy inequality due to Case \cite[Theorem 1.1]{Ca}, it follows under the assumption \eqref{eq_eig} that if one defines the functionals
\begin{align}
\overline{I}_{\hh}^{\ga}[U] &= {\kappa_{\ga} \int_X (\rho^{1-2\ga}|\nabla U|_{\bg}^2 + E(\rho)U^2)\, dv_{\bg} \over (\int_M |U|^{2n \over n-2\ga} dv_{\hh})^{n-2\ga \over n}}, \label{eq_yamabe_q}\\
\widetilde{I}_{\hh}^{\ga}[U] &= {\kappa_{\ga} \int_X (\rho^*)^{1-2\ga}|\nabla U|_{\bg}^2 dv_{\bg} + \int_M Q^{\ga}_{\hh} U^2 dv_{\hh} \over (\int_M |U|^{2n \over n-2\ga} dv_{\hh})^{n-2\ga \over n}} \nonumber
\end{align}
for $U \in W^{1,2}(X, \rho^{1-2\ga})$ such that $U \ne 0$ on $M$ (with a suitable modification for the $\ga = 1/2$ case), and values
\begin{align*}
\overline{\Lambda}^{\ga}(X,[\hh]) &= \inf\left\{\overline{I}_{\hh}^{\ga}[U]: U \in W^{1,2}(X, \rho^{1-2\ga}), U \ne 0 \text{ on } M \right\},\\
\widetilde{\Lambda}^{\ga}(X,[\hh]) &= \inf\left\{\widetilde{I}_{\hh}^{\ga}[U]: U \in W^{1,2}(X, \rho^{1-2\ga}), U \ne 0 \text{ on } M \right\},
\end{align*}
then
\[\Lambda^{\ga}(M,[\hh]) = \overline{\Lambda}^{\ga}(X,[\hh]) = \widetilde{\Lambda}^{\ga}(X,[\hh]) > -\infty.\]
Besides it is shown in \cite{GQ} that the sign of $c$ in \eqref{eq_yamabe} is the same as that of $\Lambda^{\ga}(M,[\hh])$ as in the local case ($\ga = 1$).

On the other hand, the Sobolev trace inequality
\begin{equation}\label{eq_S_trace}
\( \int_{\mr^n} |U(\bx,0)|^{2n \over n-2\ga} d\bx \)^{n-2\ga \over n}
\le S_{n,\ga} \int_0^{\infty}\int_{\mr^n} x_{n+1}^{1-2\ga} |\nabla U(\bx,x_n)|^2 d\bx dx_{n+1}
\end{equation}
is true for all functions $U \in W^{1,2}(\mr^{n+1}_+,x_{n+1}^{1-2\ga})$,
and the equality is attained by $U = cW_\ls$ for any $c \in \mr,\ \lambda > 0$ and $\sigma \in \mr^n = \partial \mr^{n+1}_+$
where $W_\ls$ are the {\it bubbles} defined as
\begin{equation}\label{eq_bubble}
\begin{aligned}
W_\ls(\bx, x_{n+1}) &= p_{n,\ga} \int_{\mr^n} {x_{n+1}^{2\ga} \over (|\bx - \by|^2+x_{n+1}^2)^{n+2\ga \over 2}} w_\ls(\by)\, d\by\\
&= g_{n,\ga} \int_{\mr^n} {1 \over (|\bx - \by|^2+x_{n+1}^2)^{n-2\ga \over 2}} w_\ls^{n+2\ga \over n-2\ga}(\by)\, d\by
\end{aligned}
\end{equation}
with
\begin{equation}\label{eq_bubble_2}
w_\ls(\bx) := \alpha_{n,\ga} \({\lambda \over \lambda^2 + |\bx - \sigma|^2}\)^{n-2\ga \over 2} = W_\ls(\bx, 0)
\end{equation}
($p_{n,\ga},\, g_{n,\ga}$ and $\alpha_{n,\ga}$ are positive numbers whose values can be found in \eqref{eq_const}).
Particularly, it holds that
\begin{equation}\label{eq_bubble_eq}
\begin{cases}
-\text{div}(x_{n+1}^{1-2\ga} \nabla W_\ls) = 0 &\text{in } \mr^{n+1}_+,\\
\pa^{\ga}_{\nu} W_\ls = - \kappa_{\ga} \(\lim\limits_{x_{n+1} \to 0+} x_{n+1}^{1-2\ga}\dfrac{\pa W_\ls}{\pa x_{n+1}}\) = (-\Delta)^{\ga} w_\ls  = w_\ls^{n+2\ga \over n-2\ga} &\text{on } \mr^n.
\end{cases} \end{equation}
(In light of the equation that $W_\ls$ solves, we say that $W_\ls$ is \textit{$\ga$-harmonic}. Refer to \cite{CS}. For future use, let $W_{\lambda} = W_{\lambda,0}$ and $w_{\lambda} = w_{\lambda,0}$.)
Moreover, if $S_{n,\ga} > 0$ denotes the best constant one can achieve in \eqref{eq_S_trace} and $(\ms^n, [g_c])$ is the standard unit $n$-dimensional sphere, then
\begin{equation}\label{eq_Lambda}
\Lambda^{\ga}(\ms^n, [g_c]) = S^{-1}_{n,\ga}\, \kappa_{\ga} = \(\int_{\mr^n} w_{\ls}^{2n \over n-2\ga} dx\)^{2\ga \over n}.
\end{equation}
Related to this fact, we have the following compactness result.
\begin{propx}\label{prop_y_exist}
Let $n > 2\ga$, $\ga \in (0,1)$ and $(X^{n+1},g^+)$ be an asymptotically hyperbolic manifold with the conformal infinity $(M^n, [\hh])$.
Also, assume that \eqref{eq_eig} is true. Then
\begin{equation}\label{eq_y_const}
-\infty < \Lambda^{\ga}(M, [\hh]) \le \Lambda^{\ga}(\ms^n, [g_c]),
\end{equation}
and the fractional Yamabe problem \eqref{eq_yamabe_e_1}-\eqref{eq_yamabe_e_2} has a positive solution if the strict inequality holds.
\end{propx}
\noindent Refer to \cite{GQ} for its proof.
Moreover since \eqref{eq_y_const} automatically holds if the $\ga$-Yamabe constant $\Lambda^{\ga}(M, [\hh])$ is negative or 0,
\textbf{we assume that $\Lambda^{\ga}(M, [\hh]) > 0$ from now on}.

\medskip
The purpose of this paper is to construct a proper nonzero test function $\Phi \in W^{1,2}(X,\rho^{1-2\ga})$
such that $0 < \overline{I}_{\hh}^{\ga}[\Phi] < \Lambda^{\ga} (\ms^n, [g_c])$ when $\gamma \in (0,1)$, $(X^{n+1}, g^+)$ is an asymptotically hyperbolic manifold, \eqref{eq_eig} holds and
\begin{itemize}
\setlength\itemsep{0em}
\item[-] $M^n$ has a point where the mean curvature $H$ is negative, $n \ge 2$ and $\ga \in (0,1/2)$; or
\item[-] $M^n$ is the non-umbilic boundary of $X^{n+1}$, $n \ge 4$ and assumption \eqref{eq_main_l_1} holds; or
\item[-] $M^n$ is the umbilic but non-locally conformally flat boundary of $X^{n+1}$, $n > 4+2\ga$ and condition \eqref{eq_main_l_2} is satisfied; or
\item[-] $X^{n+1}$ is Poincar\'e-Einstein and either $M^n$ is locally conformally flat or $n = 2$.
\end{itemize}
Then Proposition \ref{prop_y_exist} would imply the existence of a positive solution to \eqref{eq_yamabe} automatically.
The natural candidate for a positive test function is certainly the standard bubble, possibly truncated.
Indeed, this is a good choice for the first case above mentioned. Nevertheless, to cover lower dimensional manifolds or locally conformally flat boundaries, it is necessary to find more accurate test functions than the truncated bubbles (cf. \cite{GQ, GW}).
To take into account the second and third situations, we shall add a correction term on the bubble by adapting the idea of Marques \cite{Ma2} and Almaraz \cite{Al}.
For the fourth case, assuming the validity of the positive mass theorem for $P^{\ga}_{\hh}$ for $\ga \in (0,1)$, we will construct an appropriate test function by utilizing Green's function.
In the local situation ($\ga = 1$), such an approach was successfully applied by Schoen \cite{Sc} who employed the classical positive mass theorem \cite{SY1, SY2, SY4}.
His idea was later extended by Escobar \cite{Es} in the work of the boundary Yamabe problem,
which has close relationship to the fractional Yamabe problem with $\ga = 1/2$.

\medskip
Our first main result reads as follows:
Let $\pi$ be the second fundamental form of $(M,\hh) \subset (\ox, \bg)$.
The boundary $M$ is called {\it umbilic} if the tensor $T := \pi - H\bg$ vanishes on $M$.
Also $M$ is {\it non-umbilic} if it possesses a point at which $T \ne 0$.
\begin{thm}\label{thm_main_l}
Suppose that $(X^{n+1}, g^+)$ is an asymptotically hyperbolic manifold, $(M, [\hh])$ is its conformal infinity and \eqref{eq_eig} holds.
Assume also that $\rho$ is a geodesic defining function of $(M, \hh)$ and $\bg = \rho^2 g^+ = d\rho^2 \oplus h_{\rho}$ near $M = \{\rho = 0\}$.
If either
\begin{itemize}
\setlength\itemsep{0em}
\item[-] $n \ge 2$, $\ga \in (0,1/2)$ and $M^n$ has a point at which the mean curvature $H$ is negative; or
\item[-] $n \ge 4$, $\ga \in (0,1)$, $M^n$ is the non-umbilic boundary of $X^{n+1}$ and
\begin{equation}\label{eq_main_l_1}
R[g^+] + n(n+1) = o(\rho^2) \quad \text{as } \rho \to 0 \text{ uniformly on } M,
\end{equation}
\end{itemize}
then the $\ga$-Yamabe problem is solvable - namely, \eqref{eq_yamabe} has a positive solution.
\end{thm}
\begin{rmk}\label{rmk_main_l}
(1) As pointed out in Gonz\'alez-Qing \cite{GQ}, we are only permitted to change the metric on the conformal infinity $M$.
Once the boundary metric $\hh$ is fixed, the geodesic boundary defining function $\rho$ and a compact metric $\bg$ on $X$ are automatically determined
by the relations $|d\rho|_{\rho^2g^+} = 1$ and $\bg = \rho^2 g^+$.
This is a huge difference between the fractional Yamabe problem (especially, with $\ga = 1/2$) and the boundary Yamabe problem
in that one has a freedom of conformal change of the metric in the whole manifold $X$ when he/she is concerned with the boundary Yamabe problem.

Due to this reason, while it is possible to make the `extrinsic' metric $H$ vanish at a point by a conformal change in the boundary Yamabe problem,
one cannot do the same thing in the setting of the fractional Yamabe problem.
This forced us to separate the cases in the statement of Theorem \ref{thm_main_l}.

\medskip \noindent (2) As a particular consequence of the previous discussion, the Ricci tensor $R_{\rho\rho}[\bg](y)$ of $(X, \bg)$ evaluated at a point $y$ on $M$ is governed by $\hh$ and \eqref{eq_main_l_1} (see Lemma \ref{lemma_rep}).
In the boundary Yamabe problem \cite{Es}, the author could choose a metric in $X$ such that the Ricci curvature $R_{ij}[\hh](y) = 0$ of $(M, \hh)$ and $R_{\rho\rho}[\bg](y) = 0$ simultaneously.

Moreover, by putting \eqref{eq_E1} and \eqref{eq_main_l_1} together, we get
\[E(\rho)= \({n-2\ga \over 4n}\) R[\bg]\, \rho^{1-2\ga} + o(\rho^{1-2\ga}) \quad \text{near } M.\]
Hence, on account of the energy expansion, \eqref{eq_main_l_1} is the very condition that makes the boundary Yamabe problem and the $1/2$-Yamabe problem identical modulo the remainder.
Refer to Subsections \ref{subsec_num_ene} and \ref{subsec_num_ene_2}.

\medskip \noindent (3) The sign of the mean curvature at a fixed point on $M$ and \eqref{eq_main_l_1} are `intrinsic' curvature conditions of an asymptotically hyperbolic manifold
in the sense that these properties are independent of the choice of a representative of the class $[\hh]$.
Refer to Lemma \ref{lemma_cur} below for its proof.
Also Lemma \ref{lemma_H=0} claims that \eqref{eq_main_l_1} implies $H = 0$ on $M$.

\medskip \noindent (4) Note also that $2 + 2\ga \in \mn$ and $\ga \in (0,1)$ if and only if $\ga = 1/2$, and the boundary Yamabe problem on non-umbilic manifolds in dimension $n = 2 + 2\ga = 3$ was covered in \cite{Ma2}.
We suspect that the strategy suggested in \cite{Ma2} can be applied for $1/2$-Yamabe problem in the same setting.

\medskip \noindent (5) Suppose that $n \in \mn$ and $\ga \in (0,1)$ satisfy $\mcc'(n,\ga) > 0$ where $\mcc'(n,\ga)$ is the quantity defined in \eqref{eq_mcc'} below.
Moreover assume that $(M^n, [\hh])$ is the conformal infinity of an asymptotic hyperbolic manifold $(X, g^+)$
such that \eqref{eq_eig} and \eqref{eq_main_l_1} hold, and the second fundamental form $\pi$ never vanishes on $M$.
Then the solution set of \eqref{eq_yamabe} (with $c > 0$) is compact in $C^2(M)$ as shown in \cite{KMW2}.
\end{rmk}

We next consider the case when the boundary $M$ is umbilic but non-locally conformally flat.
\begin{thm}\label{thm_main_l_2}
Suppose that $n > 4+2\ga$, $\ga \in (0, 1)$ (that is, either $n \ge 6$ and $\ga \in (0,1)$, or $n = 5$ and $\ga \in (0,1/2)$) and
$(X^{n+1}, g^+)$ is an asymptotic hyperbolic manifold such that \eqref{eq_eig} holds.
Furthermore, assume that $(M^n, [\hh])$ is the umbilic boundary of $X^{n+1}$ and there is a point $y \in M$ such that the Weyl tensor $W[\hh]$ on $M$ is nonzero at $y$.
If \begin{equation}\label{eq_main_l_2}
\begin{cases}
R[g^+] + n(n+1) = o(\rho^4),\\
\pa_{\bx}^m \(R[g^+] + n(n+1)\) = o(\rho^2) \quad (m = 1, 2),\\
\pa_\rho^m\(R[g^+] + n(n+1)\) = o(\rho^2) \quad (m = 1, 2)
\end{cases}
\end{equation}
as $\rho \to 0$ uniformly on $M$, then the $\ga$-Yamabe problem is solvable. Here $\bx$ is a coordinate on $M$.
\end{thm}
\begin{rmk}
(1) As we will see later, the main order of the energy for the fractional Yamabe problem \eqref{eq_yamabe} is $\ep^4$ on an umbilic but non-locally conformal flat boundary $M$,
while it is $\ep^2$ on a non-umbilic boundary (see \eqref{eq_l_0}, \eqref{eq_l_1}, \eqref{eq_um_0} and \eqref{eq_um_1}).
Therefore it is natural to expect that Theorem \ref{thm_main_l_2} should require that $R[g^+] + n(n+1)$ decays $\rho^2$-faster than Theorem \ref{thm_main_l} near $M$.
Compare \eqref{eq_main_l_1} and \eqref{eq_main_l_2}.

Assumption \eqref{eq_main_l_2} is responsible for determining all the values of quantities which emerge in the coefficient of $\ep^4$ in the energy
(such as $R_{,ii}[\bg](y)$ and $R_{NN,ii}[\bg](y)$ - see Lemma \ref{lemma_rep_2})
and controlling the term $(n(n+1)+R[g^+]) \rho^{-2}$ in $E(\rho)$ to be ignorable.

\medskip \noindent (2) In light of Lemmas \ref{lemma_cur} and \ref{lemma_H=0}, condition \eqref{eq_main_l_2} is again intrinsic and sufficient to deduce that $H = 0$ on $M$.
Moreover every Poincar\'e-Einstein manifold satisfies \eqref{eq_main_l_2}.

\medskip \noindent (3) It is notable that $4 + 2\ga \in \mn$ and $\ga \in (0,1)$ if and only if $\ga = 1/2$, and
the boundary Yamabe problem for $n = 4 + 2\ga = 5$ was studied in \cite{Al}.
Hence it is natural to ask whether one can extend Theorem \ref{thm_main_l_2} for $\ga = 1/2$ and $n = 5$ by following the perturbation argument given in \cite{Al}.
\end{rmk}

In order to describe the last result, we first have to take into account of Green's function under our setting.
\begin{prop}\label{prop_G_exp}
Suppose that all the hypotheses of Theorem \ref{thm_ext} hold true (including \eqref{eq_eig}) and $H = 0$ on $M$.
In addition, assume further that $\Lambda^{\ga}(M, [\hh]) > 0$.
Then for each $y \in M$, there exists \textit{Green's function} $G(x,y)$ on $\ox \setminus \{y\}$ which satisfies
\begin{equation}\label{eq_G}
\begin{cases}
-\textnormal{div}_{\bg}\(\rho^{1-2\ga}\nabla G(\cdot,y)\) + E(\rho)\, G(\cdot,y) = 0 &\text{in } (X, \bg),\\
\pa^{\ga}_{\nu} G(\cdot,y) = \delta_y &\text{on } (M, \hh)
\end{cases}
\end{equation}
in the distribution sense where $\delta_y$ is the Dirac measure at $y$.
The function $G$ is unique and positive on $\ox$.
\end{prop}
\noindent The proof is postponed until Subsection \ref{subsec_Green}.
The readers may compare the above result with Guillarmou-Qing \cite{GQ2}.
Based on the previous proposition and the fact that
\[G(x,\by) = {g_{n,\ga} \over |(\bx-\by, x_{n+1})|^{n-2\ga}} \quad \text{for all } (\bx, x_{n+1}) \in \mr^{n+1}_+ \text{ and } \by \in \mr^n\]
if $(X, \bg)$ is the Poincar\'e half-plane $(\mr^{n+1}_+, x_{n+1}^{-2} dx)$, we conjecture the following.
\begin{conj}\label{conj_pos}[Positive mass theorem]
Assume that $\ga \in (0,1)$, $n > 2\ga$ and $(X^{n+1}, g^+)$ is Poincar\'e-Einstein.
Also suppose that $\Lambda^{\ga}(M, [\hh]) > 0$ and either $(M^n, [\hh])$ is locally conformally flat or $n = 2$.
Then we have an asymptotic expansion on $G(\cdot, y)$ of the form
\begin{equation}\label{eq_G_1}
G(x,y) = g_{n,\ga}\, d_{\bg}(x,y)^{-(n-2\ga)} + A + \Psi(d_{\bg}(x,y)) \quad \text{with } A \ge 0
\end{equation}
for any $x \in \ox$ near $y \in M$, where $g_{n,\ga} > 0$ is a constant appeared in \eqref{eq_bubble}
and $\Psi$ is a function in a small closed neighborhood $\mcn \subset \overline{\mr_+^{n+1}}$ of $0$ such that
\begin{equation}\label{eq_Psi}
\Psi(0) = 0 \quad \text{and} \quad \|\Psi\|_{C^{\vt_1}(\mcn)} + \left\|\nabla_{\bx} \Psi \right\|_{C^{\vt_1}(\mcn)}
+ \left\| x_{n+1}^{1-2\ga} {\pa \Psi \over \pa x_{n+1}} \right\|_{C^{\vt_1}(\mcn)} \le C
\end{equation}
for some $\vt_1 \in (0,1)$.
Furthermore, $A = 0$ if and only if $(X^{n+1},\bg)$ is conformally diffeomorphic to the standard unit ball $\mb^{n+1}$ (which we denote by $(X^{n+1},\bg) \simeq \mb^{n+1}$).
\end{conj}
\noindent Our expectation on the regularity \eqref{eq_Psi} of $\Psi$ is based on the fact that $\Psi$ is `approximately' $\ga$-harmonic near $y$.
Now we can state our third main theorem.
\begin{thm}\label{thm_main}
Suppose that $\ga \in (0,1)$, $n > 2\ga$ and $(X^{n+1}, g^+)$ is a Poincar\'e-Einstein manifold with conformal infinity $(M^n, [\hh])$.
Let $\rho$ be a geodesic defining function for $(M, \hh)$ and $\bg = \rho^2 g^+$.
If \eqref{eq_eig} holds, Conjecture \ref{conj_pos} is valid, and either $M^n$ is locally conformally flat or $n = 2$, then the fractional Yamabe problem is solvable.
\end{thm}
\begin{rmk}\label{rmk_sec_fund}
(1) Let us set a 2-tensor
\[F = \rho \(\text{Ric}[g^+] + ng^+\) \quad \text{in } X,\]
which is identically 0 if $(X, g^+)$ is Poincar\'e-Einstein.
As a matter of the fact, if $M$ is locally conformally flat, the only property of the tensor $F$ necessary to derive Theorem \ref{thm_main}
is that $\pa_\rho^m F|_{\rho = 0} = 0$ for $m = 0, \cdots, n-1$ (refer to Lemma \ref{lemma_h^m}).
We guess that \eqref{eq_G_1}, \eqref{eq_Psi}, and the condition on $A$ are still valid under this assumption.
Similarly, for the case $n = 2$, the assumption $\pa_\rho^m F|_{\rho = 0} = 0$ for $m = 0, 1$ would suffice.

\medskip \noindent (2) Since $(X^{n+1}, g^+)$ is Poincar\'e-Einstein, the second fundamental form on $M$ is trivial. Thus the mean curvature $H$ on $M$ vanishes and $M$ is umbilic.
\end{rmk}

This paper is organized as follows:
In Section \ref{sec_Mar}, we establish Theorem \ref{thm_main_l} by intensifying the ideas of Marques \cite{Ma2} and Gonz\'alez-Qing \cite{GQ}.
Section \ref{sec_Al} provides the proof of Theorem \ref{thm_main_l_2} which further develops the approach of Almaraz \cite{Al} and Gonz\'alez-Wang \cite{GW}.
In Section \ref{sec_lcf}, Theorem \ref{thm_main} is achieved under the validity of the positive mass theorem.
In particular, Subsection \ref{subsec_Green} is devoted to investigate the existence, uniqueness, positivity of Green's function (i.e. Proposition \ref{prop_G_exp}).
Then we are concerned with the case that $M$ is locally conformally flat (in Subsection \ref{subsec_lcf}) and 2-dimensional (in Subsection \ref{subsec_2_dim}).
Finally, we examine the asymptotic behavior of the bubble $W_{1,0}$ near infinity in Appendix \ref{sec_app},
and compute some integrations regarding $W_{1,0}$ which are needed in the energy expansions in Appendix \ref{sec_app_2}.

\bigskip \noindent \textbf{Notations.}

\medskip \noindent - The Einstein convention is used throughout the paper. The indices $i,\, j,\, k$ and $l$ always take values from $1$ to $n$, and $a$ and $b$ range over values from $1$ to $n+1$.

\medskip \noindent - For a tensor $T$, notations $T_{;a}$ and $T_{,a}$ indicate covariant differentiation and partial differentiation of $T$, respectively.

\medskip \noindent - For a tensor $T$ and a number $q \in \mn$, we use
\[\text{Sym}_{i_1 \cdots i_q} T_{i_1 \cdots i_q} = {1 \over q!} \sum_{\sigma \in S_q} T_{i_{\sigma(1)} \cdots i_{\sigma(q)}}\]
where $S_q$ is the group of all permutations of $q$ elements.

\medskip \noindent - We denote $N = n+1$. Also, for $x \in \mr^N_+ := \{(x_1, \cdots, x_n, x_N) \in \mr^N: x_N > 0\}$,
we write $\bx = (x_1, \cdots, x_n, 0) \in \partial \mr^N_+ \simeq \mr^n$ and $r = |\bx|$.

\medskip \noindent - For $n > 2\ga$, we set $p = (n+2\ga)/(n-2\ga)$.

\medskip \noindent - For any $\vr > 0$, $B^n(0,\vr)$ and $B^N_+(0,\vr)$ are the $n$-dimensional ball and the $N$-dimensional upper half-ball centered at 0 whose radius is $\vr$, respectively.

\medskip \noindent - $|\ms^{n-1}|$ is the surface area of the $(n-1)$-dimensional unit sphere $\ms^{n-1}$.

\medskip \noindent - For any $t \in \mr$, let $t_+ = \max\{0, t\} \ge 0$ and $t_- = \max\{0, -t\} \ge 0$ so that $t = t_+ - t_-$.

\medskip \noindent - The following positive constants are given in \eqref{eq_yamabe_e_2}, \eqref{eq_bubble} and \eqref{eq_bubble_2}:
\begin{equation}\label{eq_const}
\kappa_{\ga} = {\Gamma(\ga) \over 2^{1-2\ga} \Gamma(1-\ga)},
\quad p_{n,\ga} = {\Gamma\({n+2\ga \over 2}\) \over \pi^{n/2}\Gamma(\ga)},
\quad g_{n,\ga} = {\Gamma\({n-2\ga \over 2}\) \over \pi^{n/2} 2^{2\ga}\Gamma(\ga)},
\quad \alpha_{n, \ga} = 2^{n-2\ga \over 2} \({\Gamma\({n+2\ga \over 2}\) \over \Gamma\({n-2\ga \over 2}\)}\)^{n-2\ga \over 4\ga}.
\end{equation}

\medskip \noindent - $C > 0$ is a generic constant which may vary from line to line.

\section{Non-minimal and Non-umbilic Conformal Infinities}\label{sec_Mar}
\subsection{Geometric Background}
We initiate this section by proving that the sign of the mean curvature, \eqref{eq_main_l_1} and non-umbilicity of a point on $M$ are intrinsic conditions.
\begin{lemma}\label{lemma_cur}
Suppose that $(X, g^+)$ be an asymptotically hyperbolic manifold with conformal infinity $(M, [\hh])$.
Moreover, let $\rho$ and $\trho$ be the geodesic boundary defining functions associated to two representatives $\hh$ and $\tilde{h}$ of the class $[\hh]$, respectively.
We also define $\bg = \rho^2 g^+$ and $\tilde{g} := \trho^2 g^+$,
denote by $\pi = -\bg_{,N}/2$ and $\tilde{\pi}$ the second fundamental forms of $(M, \hh) \subset (\ox, \bg)$ and $(M, \tilde{h}) \subset (\ox, \tilde{g})$, respectively,
and set $H = \bg^{ij}\pi_{ij}/n$ and $\widetilde{H} = \tilde{g}^{ij} \tilde{\pi}_{ij}/n$.
Then we have
\begin{equation}\label{eq_cur_1}
C^{-1} \le {\trho \over \rho} \le C \quad \text{in } X \quad \text{and} \quad H = \left.\({\trho \over \rho}\)\right|_{\rho = 0} \widetilde{H} \quad \text{on } M
\end{equation}
for some $C > 1$. Furthermore if $H = 0$ on $M$, then
\begin{equation}\label{eq_cur_2}
\pi = \left.\({\rho \over \trho}\)\right|_{\rho = 0} \tilde{\pi} \quad \text{on } M.
\end{equation}
\end{lemma}
\begin{proof}
The assertion on $H$ in \eqref{eq_cur_1} is proved in \cite[Lemma 2.3]{GQ}.
For the first inequality in \eqref{eq_cur_1}, it suffices to observe that $\trho/\rho$ is bounded above and bounded away from 0 near $M$.
Indeed, this follows from the fact that
\[\tilde{h} = \tilde{g}|_M = \trho^2 g^+|_M = \({\trho \over \rho}\)^2 \bg|_M = \({\trho \over \rho}\)^2 \hh \quad \text{on } M.\]

Let us define tensors $T = \pi - H\bg$ and $\widetilde{T} = \tilde{\pi} - \widetilde{H} \tilde{g}$ on $M$.
Then we see from \cite[Proposition 1.2]{Es2} that
\[\tilde{\pi} = \widetilde{T} = \({\trho \over \rho}\)T = \({\trho \over \rho}\)\pi \quad \text{on } M\]
provided $H = 0$ on $M$, which confirms \eqref{eq_cur_2}.
\end{proof}

Given any fixed point $y \in M$, let $\bx = (x_1, \cdots, x_n)$ be normal coordinates on $M$ at $y$ (identified with $0$) and $x_N = \rho$.
In other words, let $x = (\bx, x_N)$ be {\it Fermi coordinates}.
The following lemma provides the expansion of the metric $\bg$ near $y = 0$. See \cite[Lemma 3.1]{Es} for its proof.
\begin{lemma}\label{lemma_metric}
Suppose that $(X, g^+)$ is an asymptotically hyperbolic manifold and $y$ is an arbitrary point on the conformal infinity $(M, [\hh])$. Then, in terms of Fermi coordinates around $y$, it holds that
\[\sqrt{|\bg|}(x) = 1 - nH x_N + {1 \over 2} \(n^2H^2 - \|\pi\|^2 - R_{NN}[\bg]\) x_N^2 - H_{,i}x_ix_N - {1 \over 6} R_{ij}[\hh] x_ix_j + O(|x|^3)\]
and
\[\bg^{ij}(x) = \delta_{ij} + 2 \pi_{ij} x_N + {1 \over 3} R_{ikjl}[\hh] x_kx_l + \bg^{ij}_{\phantom{ij},Nk} x_Nx_k + (3\pi_{ik}\pi_{kj} + R_{iNjN}[\bg]) x_N^2 + O(|x|^3)\]
near $y$ (identified with a small half-ball $B^N_+(0, 2\eta_0)$ near 0 in $\mr_+^N$).
Here $\|\pi\|^2 = \hh^{ik}\hh^{jl}\pi_{ij}\pi_{kl}$ is the square of the norm of the second fundamental form $\pi$ on $(M, \hh) \subset (\ox, \bg)$,
$R_{ikjl}[\hh]$ is a component of the Riemannian curvature tensor on $M$, $R_{iNjN}[\bg]$ is that of the Riemannian curvature tensor in $X$,
$R_{ij}[\hh] = R_{ikjk}[\hh]$ and $R_{NN}[\bg] = R_{iNiN}[\bg]$.
Every tensor in the expansions is computed at $y = 0$.
\end{lemma}

Now notice that the transformation law of the scalar curvature (see (1.1) of \cite{Es}) implies
\begin{equation}\label{eq_R_gp}
R[g^+] + n(n+1) = 2n \({\pa_{\rho} \sqrt{|\bg|} \over \sqrt{|\bg|}}\) \rho + R[\bg] \rho^2.
\end{equation}
It readily shows that \eqref{eq_main_l_1} and \eqref{eq_main_l_2} indicate $H = 0$ on $M$.
\begin{lemma}\label{lemma_H=0}
Suppose that $(X, g^+)$ is an asymptotically hyperbolic manifold with conformal infinity $(M, [\hh])$.
If $R[g^+] + n(n+1) = o(\rho)$ as $\rho \to 0$, then $H = 0$ on $M$.
\end{lemma}
\begin{proof}
Fix any $y \in M$. By \eqref{eq_R_gp}, we have
\[o(1) = 2n \({\pa_{\rho} \sqrt{|\bg|}(y) \over \sqrt{|\bg|}(y)}\) + R[\bg](y) \rho + o(1) = -2n^2 H(y) + o(1)\]
as a point tends to $y$.
This implies $H(y) = 0$, and therefore the assertion follows.
\end{proof}

We next select a good background metric on $X$ under the validity of hypothesis \eqref{eq_main_l_1}.
\begin{lemma}\label{lemma_rep}
Let $(X, g^+)$ be an asymptotically hyperbolic manifold such that condition \eqref{eq_main_l_1} holds.
Then the conformal infinity $(M, [\hh])$ admits a representative $\hh \in [\hh]$, the geodesic boundary defining function $\rho$ and the metric $\bg = \rho^2 g^+$ satisfying
\begin{equation}\label{eq_rep}
H = 0 \text{ on } M, \quad R_{ij}[\hh](y) = 0 \quad \text{and} \quad R_{\rho\rho}[\bg](y) = {1-2n \over 2(n-1)} \|\pi(y)\|^2
\end{equation}
for a fixed point $y \in M$.
\end{lemma}
\begin{proof}
According to \cite[Theorem 5.2]{LP}, one may choose a representative $\hh$ of the conformal class $[\hh]$ such that $R_{ij}[\hh](y) = 0$.
Besides Lemmas \ref{lemma_H=0} and \ref{lemma_cur} assure that $H = 0$ on $M$ for any $\hh \in [\hh]$.
Hence assumption \eqref{eq_main_l_1} can be interpreted as
\begin{align*}
o(1) &= 2n \({\pa_{\rho} \sqrt{|\bg|} \over \rho \sqrt{|\bg|} }\) + R[\bg] = {n \over \rho} \bg^{ab}\bg_{ab,\rho} + R[\bg]
= n \(\bg^{ab}_{\phantom{ab},\rho}\, \bg_{ab,\rho} + \bg^{ab} \bg_{ab,\rho\rho}\) + R[\bg] + o(1) \\
&= -2n\(R_{\rho\rho}[\bg] + \|\pi\|^2\) + \(2R_{\rho\rho}[\bg] + \|\pi\|^2 + R[\hh] - H^2\) + o(1)
\end{align*}
as $\rho \to 0$ where we used $H = 0$ on $M$ for the third equality and the Gauss-Codazzi equation for the fourth equality (see the proof of Lemmas 3.1 and 3.2 of \cite{Es}).
Taking the limit to $y \in M$, we get
\[0 = 2(1-n)R_{\rho\rho}[\bg](y) + (1-2n)\|\pi(y)\|^2.\]
The third equality of \eqref{eq_rep} is its direct consequence.
\end{proof}

Lastly, we recall the function $E$ in \eqref{eq_E} and \eqref{eq_E1}. In a collar neighborhood of $M$ where $\rho = x_N$, it can be seen that
\begin{equation}\label{eq_E_2}
E(x_N) = \({n-2\ga \over 4n}\) \left[ R[\bg] - (n(n+1)+R[g^+]) x_N^{-2} \right] x_N^{1-2\ga} = -\({n-2\ga \over 2}\) \({\pa_N \sqrt{|\bg|} \over \sqrt{|\bg|}}\) x_N^{-2\ga}
\end{equation}
where the second equality holds because of \eqref{eq_R_gp}.

\subsection{Non-minimal Conformal Infinity}\label{subsec_nm_ene}
Let $y \in M$ be a point identified with $0 \in \mr^n$ such that $H(y) < 0$ and $B^N_+(0, 2\eta_0) \subset \mr_+^N$ its neighborhood which appeared in Lemma \ref{lemma_metric}.
Also, we select any smooth radial cut-off function $\psi \in C^{\infty}_c(\mr_+^N)$ such that $\psi = 1$ in $B^N_+(0, \eta_0)$ and 0 in $\mr_+^N \setminus B^N_+(0, 2\eta_0)$.
In this subsection, we shall show that $\overline{I}_{\hh}^{\ga}[\psi W_{\ep}] < \Lambda^{\ga}(\ms^n, [g_c])$ for any $n \ge 2$ and $\ga \in (0,1/2)$ where $W_{\ep} = W_{\ep,0}$ as before.

\medskip
Before starting the computation, let us make one useful observation:
Assume that $n > m + 2\ga$ for a certain $m \in \mn$.
Then we get from \eqref{eq_W_dec_2} and \eqref{eq_W_dec_3} that
\begin{equation}\label{eq_rem}
\int_{B^N_+(0, \eta_0)} x_N^{1-2\ga} |x|^{m+1} |\nabla W_\ep|^2 dx = \eta_0^{m-\zeta} \int_{B^N_+(0, \eta_0)} x_N^{1-2\ga} |x|^{m + \zeta} |\nabla W_\ep|^2 dx = O(\ep^{m + \zeta}) = o(\ep^m)
\end{equation}
by choosing a small number $\zeta > 0$ such that $n > m + 2\ga + \zeta$.

\begin{prop}\label{prop_nonm}
Suppose that $(X^{n+1},g^+)$ is an asymptotically hyperbolic manifold with conformal infinity $(M, [\hh])$ and $y \in M$ be a point such that $H(y) < 0$.
Then for any $\ep > 0$ small, $n \ge 2$ and $\ga \in (0,1/2)$, we have
\begin{equation}\label{eq_nonm}
\begin{aligned}
\overline{I}_{\hh}^{\ga}[\psi W_{\ep}] &\le \Lambda^{\ga}(\ms^n, [g_c]) + \ep \underbrace{\left[{2n^2-2n+1-4\ga^2 \over 2(1-2\ga)}\right]
\, \left[{\kappa_{\ga} \int_{\mr^N_+} x_N^{2-2\ga} |\nabla W_1|^2 dx \over \int_{\mr^n} w_1^{p+1} dx}\right]}_{>0} \, H(y) + o(\ep) \\
& < \Lambda^{\ga}(\ms^n, [g_c])
\end{aligned}
\end{equation}
where $\overline{I}_{\hh}^{\ga}$ is the $\ga$-Yamabe functional given in \eqref{eq_yamabe_q}, and $\Lambda^{\ga}(\ms^n, [g_c])$ and $\kappa_{\gamma}$ are positive constants introduced in \eqref{eq_Lambda} and \eqref{eq_const}.
\end{prop}
\begin{proof}
Since the proof is essentially the same as that of \cite[Proposition 6.1]{CK}, we briefly sketch it.
By Lemma \ref{lemma_metric} and \eqref{eq_rem}, we discover
\begin{multline*}
\int_{B^N_+(0, \eta_0)} x_N^{1-2\ga} |\nabla W_\ep|_{\bg}^2 dv_{\bg} \\
= \int_{B^N_+(0, \eta_0)} x_N^{1-2\ga} |\nabla W_1|^2 dx + \ep H \(2 \int_{\mr^N_+} x_N^{2-2\ga} |\nabla_{\bx} W_1|^2 dx - n \int_{\mr^N_+} x_N^{2-2\ga} |\nabla W_1|^2 dx\) + o(\ep)
\end{multline*}
and
\[\int_M (\psi W_{\ep})^{p+1} dv_{\hh} = \int_{B^n(0,\eta_0)} w_\ep^{p+1} \(1 + O(|\bx|^2)\) d\bx + O(\ep^n) = \int_{\mr^n} w_1^{p+1} dx + o(\ep).\]
Moreover, according to Lemma \ref{lemma_metric} and \eqref{eq_E_2}, we have
\[\int_{B^N_+(0, \eta_0)} E(x_N) W_{\ep}^2 dv_{\bg} = \left[{n(n-2\ga) \over 2}\right] \ep H \int_{\mr^N_+} x_N^{-2\ga} W_1^2 dx + o(\ep).\]
Thus the above estimates and Lemma \ref{lemma_int_0} confirm \eqref{eq_nonm}.
\end{proof}
\noindent Unlike the other existence results to be discussed later, we need to assume that $\ga \in (0,1/2)$ for Proposition \ref{prop_nonm}.
Such a restriction is necessary in two reasons:
First of all, $\ga \in (0,1/2)$ is necessary for the function $x_N^{-2\ga} W_1^2$ to be integrable in $\mr^N_+$.
Secondly the mean curvature $H$ should vanish for $\ga \in (1/2,1)$ to guarantee the validity of the extension theorem (Theorem \ref{thm_ext}).

\subsection{Non-umbilic Conformal Infinity: Higher Dimensional Cases}\label{subsec_num_ene}
We fix a non-umbilic point $y = 0 \in M$.
Let also $B^N_+(0, 2\eta_0) \subset \mr_+^N$ be a small neighborhood of $0$ and $\psi \in C^{\infty}_c(B^N_+(0, 2\eta_0))$ a cut-off function chosen in the previous subsection.

\begin{lemma}\label{lemma_h}
Let $J^{\ga}_{\hh}$ be the energy functional defined as
\begin{equation}\label{eq_J}
J^{\ga}_{\hh}[U; X] = \int_X (\rho^{1-2\ga}|\nabla U|_{\bg}^2 + E(\rho)U^2)\, dv_{\bg} \quad \text{for any } U \in W^{1,2}(X, \rho^{1-2\ga}).
\end{equation}
Assume also that \eqref{eq_rep} holds.
Then for any $\ep > 0$ small, $n > 2 + 2\ga$ and $\ga \in (0,1)$, it is valid that
\begin{multline}\label{eq_J_W}
J^{\ga}_{\hh} \left[\psi W_{\ep}; B^N_+(0, \eta_0)\right]
= \int_{B^N_+(0, \eta_0)} x_N^{1-2\ga} |\nabla W_1|^2 dx \\
+ \ep^2 \|\pi\|^2 \left[-\({1+b \over 2}\) \mcf_2 + \({3+b \over n}\) \mcf_3 + \({n-2\ga \over 2}\) (1+b) \mcf_1 \right] + o(\ep^2)
\end{multline}
where $b := (1-2n)/(2n-2)$, $\|\pi\|$ is the norm of the second fundamental form at $y = 0 \in M$, and the values $\mcf_1,\, \mcf_2$ and $\mcf_3$ are given in Lemma \ref{lemma_int}.
\end{lemma}
\begin{proof}
We borrow the argument presented in \cite[Theorem 1.5]{GQ}.
According to Lemma \ref{lemma_metric} and \eqref{eq_rep}, there holds that
\begin{equation}\label{eq_met}
\sqrt{|\bg|}(\bx, x_N) = 1 - \({1+b \over 2}\) \|\pi\|^2 x_N^2 + O(|(\bx, x_N)|^3) \quad \text{in } B^N_+(0, \eta_0).
\end{equation}
Hence we obtain with \eqref{eq_rem} that
\begin{multline*}
\int_{B^N_+(0, \eta_0)} x_N^{1-2\ga} |\nabla W_\ep|_{\bg}^2 dv_{\bg}
= \int_{\mr^N_+} x_N^{1-2\ga} |\nabla W_\ep|^2 dx \\
+ \ep^2 \left[ (3\pi_{ik}\pi_{kj} + R_{iNjN}[\bg]) \int_{\mr^N_+} x_N^{3-2\ga} \pa_i W_1 \pa_j W_1 dx 
- \({1+b \over 2}\) \|\pi\|^2 \int_{\mr^N_+} x_N^{3-2\ga} |\nabla W_1|^2 dx \right] + o(\ep^2).
\end{multline*}
Also, in view of \eqref{eq_E_2} and \eqref{eq_met},
\[E(x_N) = \({n-2\ga \over 2}\) (1+b) \|\pi\|^2 x_N^{1-2\ga} + O(|x|^2x_N^{-2\ga})\]
for $x_N \ge 0$ small, so
\[\int_{B^N_+(0, \eta_0)} E(x_N) W_{\ep}^2 dv_{\bg} = \ep^2 \({n-2\ga \over 2}\) (1+b) \|\pi\|^2 \int_{\mr^N_+} x_N^{1-2\ga} W_1^2 dx + o(\ep^2).\]
Collecting every calculation, we discover \eqref{eq_J_W}.
\end{proof}

The previous lemma ensures the existence of a positive solution to \eqref{eq_yamabe} for non-umbilic conformal infinity $M^n$ with $n \in \mn$ sufficiently high.
\begin{cor}
Assume that $(X^{n+1},g^+)$ is an asymptotically hyperbolic manifold
and $\hh$ is the representative of the conformal infinity $M$ found in Lemma \ref{lemma_metric}.
If $n > 2 + 2\ga$ and $\ga \in (0,1)$, we have that
\begin{equation}\label{eq_l_0}
\overline{I}_{\hh}^{\ga}[\psi W_{\ep}] \le \Lambda^{\ga}(\ms^n, [g_c]) - \ep^2 \mcc'(n, \ga)\, \Lambda^{\ga}(\ms^n, [g_c])^{-{n-2\ga \over 2\ga}} \kappa_{\ga}\, |\ms^{n-1}| A_3B_2 \|\pi\|^2 + o(\ep^2)
\end{equation}
where the positive constants $\Lambda^{\ga}(\ms^n, [g_c])$, $\kappa_{\ga}$, $A_3$ and $B_2$ are introduced in \eqref{eq_Lambda}, \eqref{eq_const} and \eqref{eq_AB}, respectively,
and $\mcc'(n, \ga)$ is the number given by
\begin{equation}\label{eq_mcc'}
\mcc'(n, \ga) = {3n^2 + n(16\ga^2-22)+20(1-\ga^2) \over 8n(n-1)(1-\ga^2)}.
\end{equation}
\end{cor}
\begin{proof}
Estimate \eqref{eq_l_0} comes from Lemmas \ref{lemma_h}, \ref{lemma_int} and the computations made in the proof of \cite[Theorem 1.5]{GQ}.
The details are left to the reader.
\end{proof}
\noindent
By \eqref{eq_cur_2}, we still have that $\pi \ne 0$ at $y \in M$ even after picking a new representative of the conformal infinity.
Furthermore, the number $\mcc'(n, \ga)$ is positive when $n \ge 4$ for $\ga > \sqrt{5/11} \simeq 0.674$, $n \ge 5$ for $\ga > 1/2$, $n \ge 6$ for $\ga > \sqrt{1/19} \simeq 0.229$ and $n \ge 7$ for any $\ga > 0$.
Hence, in this regime, one is able to deduce the existence of a positive solution of \eqref{eq_yamabe} by testing the truncated standard bubble into the $\ga$-Yamabe functional.

\subsection{Non-umbilic Conformal Infinity: Lower Dimensional Cases}\label{subsec_num_ene_2}
We remind the non-umbilic point $y \in M$ identified with the origin of $\mr^N_+$, the small number $\eta_0 > 0$ and the cut-off function $\psi \in C^{\infty}_c(\mr^N_+)$. Furthermore, we introduce
\begin{equation}\label{eq_Psi_ep}
\Psi_{\ep}(\bx, x_N) = M_1 \pi_{ij} x_ix_j x_N r^{-1} \pa_r W_{\ep} = \ep \cdot \ep^{-{n-2\ga \over 2}} \Psi_1(\ep^{-1}\bx, \ep^{-1} x_N)
\end{equation}
for each $\ep > 0$ where $M_1 \in \mr$ is a number to be determined later, $\pi_{ij}$'s are the coefficients of the second fundamental form at $y$ and $r = |\bx|$.
Our ansatz to deal with lower dimensional cases is defined by
\[\Phi_{\ep} := \psi(W_{\ep} + \Psi_{\ep}) \quad \text{in } X.\]
The definition of $\Phi_{\ep}$ is inspired by \cite{Ma2}.

\medskip
The main objective of this subsection is to prove
\begin{prop}\label{prop_l}
Suppose that $(X^{n+1},g^+)$ is an asymptotically hyperbolic manifold.
Moreover $\hh$ is the representative of the conformal infinity $M$ satisfying \eqref{eq_rep}.
If $n > 2 + 2\ga$ and $\ga \in (0,1)$, we have that
\begin{equation}\label{eq_l_1}
\overline{I}_{\hh}^{\ga}[\Phi_{\ep}] \le \Lambda^{\ga}(\ms^n, [g_c]) - \ep^2 \mcc(n, \ga)\, \Lambda^{\ga}(\ms^n, [g_c])^{-{n-2\ga \over 2\ga}} \kappa_{\ga}\, |\ms^{n-1}| A_3B_2 \|\pi\|^2 + o(\ep^2)
\end{equation}
where $\mcc(n, \ga)$ is the number defined by
\[\mcc(n, \ga) = {3n^2 + n(16\ga^2-22)+20(1-\ga^2) \over 8n(n-1)(1-\ga^2)} + {16(n-1)(1-\ga^2) \over n(3n^2+n(2-8\ga^2)+4\ga^2-4)}.\]
\end{prop}
\noindent It can be checked that $\mcc(n, \ga) > 0$ whenever $n \ge 4$ and $\ga \in (0,1)$.
Thus the above proposition justifies the statement of Theorem \ref{thm_main_l}.
While we have $\mcc(3, \ga) > 0$ for $\ga > 1/2$, it also holds that $n > 2+2\ga > 3$.
Therefore we get no result for $n = 3$.

\begin{proof}[Proof of Proposition \ref{prop_l}]
The proof consists of 3 steps.

\medskip \noindent \textsc{Step 1 (Energy in the half-ball $B^N_+(0,\eta_0)$).}
Since $\psi = 1$ in $B^N_+(0,\eta_0)$, we discover
\begin{equation}\label{eq_l_5}
\begin{aligned}
&\ J^{\ga}_{\hh}\left[\psi(W_{\ep} + \Psi_{\ep}); B^N_+(0,\eta_0)\right] \\
&= J^{\ga}_{\hh} \left[\psi W_{\ep}; B^N_+(0, \eta_0)\right]
+ 2 \int_{B^N_+(0,\eta_0)} x_N^{1-2\ga} \la\nabla W_{\ep}, \nabla \Psi_{\ep}\ra_{\bg} \, dv_{\bg}
+ \int_{\mr^N_+} x_N^{1-2\ga} |\nabla \Psi_{\ep}|^2 dx + o(\ep^2)
\end{aligned}
\end{equation}
where the functional $J^\ga_{\hh}$ is defined in \eqref{eq_J}.
Moreover, we note from Lemma \ref{lemma_metric} that the mean curvature $H = \pi_{ii}/n$ vanishes at the origin, which yields
\begin{equation}\label{eq_l_4}
\begin{aligned}
&\ \int_{B^N_+(0,\eta_0)} x_N^{1-2\ga} \nabla W_{\ep} \cdot \nabla \Psi_{\ep} \, dx \\
&= \ep \, M_1 \, \int_{B^N_+(0,\eta_0/\ep)} x_N^{2-2\ga} \pi_{ij}x_ix_j \left[2 r^{-2} (\pa_r W_1)^2 + r \pa_r(r^{-1}\pa_r W_1) \right] dx \\
&\ + \ep M_1 \int_{B^N_+(0,\eta_0/\ep)} x_N^{1-2\ga} \pi_{ij}x_ix_j r^{-1} (\pa_N W_1) \left[ (\pa_r W_1) + x_N (\pa_{Nr} W_1) \right] dx \\
&= 0.
\end{aligned}
\end{equation}
Hence we obtain from the definition \eqref{eq_Psi_ep} of $\Psi_{\ep}$ and \eqref{eq_l_4} that
\begin{equation}\label{eq_l_6}
\begin{aligned}
&\ 2 \int_{B^N_+(0,\eta_0)} x_N^{1-2\ga} \la\nabla W_{\ep}, \nabla \Psi_{\ep}\ra_{\bg} \, dv_{\bg} \\
&= 2 \int_{B^N_+(0,\eta_0)} x_N^{1-2\ga} \nabla W_{\ep} \cdot \nabla \Psi_{\ep} \, dx
+ 4 \pi_{ij} \int_{\mr^N_+} x_N^{2-2\ga} \pa_i W_{\ep}\, \pa_j \Psi_{\ep}\, dx + o(\ep^2) \\
&= \ep^2 4M_1 \pi_{ij} \int_{\mr^N_+} x_N^{3-2\ga} x_i \left[2\pi_{jk} x_k  r^{-2} (\pa_r W_1)^2 + \pi_{kl}x_kx_lx_j r^{-2} (\pa_r W_1)\, \pa_r(r^{-1} \pa_r W_1) \right] dx + o(\ep^2) \\
&= \ep^2 4M_1 \left[{2 \over n} \mcf_3 + {2 \over n(n+2)} \(-\mcf_3 + \mcf_4\) \right] \|\pi\|^2 + o(\ep^2) \\
&= \ep^2 \({4 \over n}\) M_1 |\ms^{n-1}| A_3B_2 \|\pi\|^2 + o(\ep^2)
\end{aligned}
\end{equation}
where the constants $\mcf_3, \mcf_4$ as well as $\mcf_1, \mcf_2, \mcf_5, \cdots, \mcf_8$ are defined in Lemma \ref{lemma_int}.
In a similar fashion, it can be found that
\begin{equation}\label{eq_l_7}
\begin{aligned}
\int_{\mr^N_+} x_N^{1-2\ga} |\nabla \Psi_{\ep}|^2 dx &= \ep^2 \left[{2M_1^2 \over n(n+2)}\right] \(\mcf_3 - 2\mcf_4 + \mcf_5 + \mcf_6 + 2\mcf_7 + \mcf_8\) \|\pi\|^2 + o(\ep^2) \\
&= \ep^2 \left[{3n^2 + 2n(1-4\ga^2) - 4(1-\ga^2) \over 4n(n-1)(1-\ga^2)}\right] M_1^2 |\ms^{n-1}| A_3B_2 \|\pi\|^2 + o(\ep^2).
\end{aligned}
\end{equation}

\medskip \noindent \textsc{Step 2 (Energy in the half-annulus $B^N_+(0,2\eta_0) \setminus B^N_+(0,\eta_0)$).}
According to \eqref{eq_W_dec}, \eqref{eq_W_dec_2} and \eqref{eq_W_dec_3} (cf. \eqref{eq_rem}), it holds
\begin{equation}\label{eq_l_8}
J^{\ga}_{\hh}\left[\psi(W_{\ep} + \Psi_{\ep}); X \setminus B^N_+(0,\eta_0) \right] = o(\ep^2).
\end{equation}
Consequently, one deduces from \eqref{eq_l_5}, \eqref{eq_l_6}-\eqref{eq_l_8} and Lemma \ref{lemma_int} that
\begin{equation}\label{eq_l_2}
J^\ga_{\hh}[\psi(W_{\ep} + \Psi_{\ep}); X] \le \int_{\mr^N_+} x_N^{1-2\ga} |\nabla W_1|^2 dx - \ep^2 \mcc(n, \ga) |\ms^{n-1}| A_3B_2 \|\pi\|^2 + o(\ep^2)
\end{equation}
by choosing the optimal $M_1 \in \mr$.

\medskip \noindent \textsc{Step 3 (Completion of the proof).}
Lemma \ref{lemma_metric} and the fact that $\Psi_{\ep} = 0$ on $M$ tell us that
\begin{equation}\label{eq_l_3}
\int_M |\psi(W_{\ep} + \Psi_{\ep})|^{p+1} dv_{\hh} = \int_{B^n(0, 2\eta_0)} (\psi w_\ep)^{p+1} (1+O(|\bx|^3))\, d\bx
\ge 
\int_{\mr^n} w_1^{p+1} d\bx + o(\ep^2).
\end{equation}
Combining \eqref{eq_l_2} and \eqref{eq_l_3} gives estimate \eqref{eq_l_1}.
The proof is concluded.
\end{proof}

\section{Umbilic Non-locally Conformally Flat Conformal Infinities}\label{sec_Al}
\subsection{Geometric Background}
For a fixed point $y \in M$ identified with $0 \in \mr^n$, let $\bx = (x_1, \cdots, x_n)$ be the normal coordinate on $M$ at $y$ and $x_N = \rho$.
The following expansion of the metric is borrowed from \cite{Ma}.
\begin{lemma}\label{lemma_metric_2}
Suppose that $(X, g^+)$ is an asymptotically hyperbolic manifold and $y$ is a point in $M$ such that \eqref{eq_rep} holds and $\pi = 0$ on $M$.
Then, in terms of normal coordinates around $y$, it holds that
\begin{equation}\label{eq_bg_exp_4}
\begin{aligned}
\sqrt{|\bg|}(\bx, x_N) &= 1 - {1 \over 12} R_{ij;k}[\hh] x_ix_jx_k - {1 \over 2} R_{NN;i}[\bg] x_N^2x_i - {1 \over 6} R_{NN;N}[\bg] x_N^3 \\
&\ - {1 \over 20} \({1 \over 2} R_{ij;kl}[\hh] + {1 \over 9} R_{miqj}[\hh] R_{mkql}[\hh]\) x_ix_jx_kx_l - {1 \over 4} R_{NN;ij}[\bg] x_N^2x_ix_j \\
&\ - {1 \over 6} R_{NN;Ni}[\bg] x_N^3x_i - {1 \over 24} \left[ R_{NN;NN}[\bg] + 2 (R_{iNjN}[\bg])^2 \right] x_N^4 + O(|(\bx, x_N)|^5)
\end{aligned}
\end{equation}
and
\begin{equation}\label{eq_b_exp}
\begin{aligned}
\bg^{ij}(\bx, x_N) &= \delta_{ij} + {1 \over 3} R_{ikjl}[\hh] x_kx_l + R_{iNjN}[\bg] x_N^2 + {1 \over 6} R_{ikjl;m}[\hh] x_kx_lx_m + R_{iNjN;k}[\bg] x_N^2x_k \\
&\ + {1 \over 3} R_{iNjN;N}[\bg] x_N^3 + \({1 \over 20} R_{ikjl;mq}[\hh] + {1 \over 15} R_{iksl}[\hh] R_{jmsq}[\hh]\) x_kx_lx_mx_q \\
&\ + \({1 \over 2} R_{iNjN;kl}[\bg] + {1 \over 3} \textnormal{Sym}_{ij}(R_{iksl}[\hh] R_{sNjN}[\bg])\)x_N^2x_kx_l + {1 \over 3} R_{iNjN;kN}[\bg] x_N^3x_k \\
&\ + {1 \over 12} \(R_{iNjN;NN}[\bg] + 8 R_{iNsN}[\bg] R_{sNjN}[\bg]\) x_N^4 + O(|(\bx, x_N)|^5)
\end{aligned}
\end{equation}
near $y$ (identified with a small half-ball $B^N_+(0, 2\eta_0)$ near 0 in $\mr_+^N$).
Here every tensors are computed at $y$ and the indices $m,\, q$ and $s$ run from 1 to $n$ as well.
\end{lemma}

To treat umbilic but non-locally conformally flat boundaries, we also need the following extension of Lemma \ref{lemma_rep}.
\begin{lemma}\label{lemma_rep_2}
For $n \ge 3$, let $(X^{n+1}, g^+)$ be an asymptotically hyperbolic manifold such that the conformal infinity $(M^n, [\hh])$ is umbilic and \eqref{eq_main_l_2} holds.
For a fixed point $y \in M$, there exist a representative $\hh$ of the class $[\hh]$,
the geodesic boundary defining function $\rho$ ($= x_N$ near $M$) and the metric $\bg = \rho^2 g^+$ such that
\begin{itemize}
\item[(1)] $R_{ij;k}[\hh](y) + R_{jk;i}[\hh](y) + R_{ki;j}[\hh](y) = 0$,
\item[(2)] $\textnormal{Sym}_{ijkl}\(R_{ij;kl}[\hh] + {2 \over 9} R_{miqj}[\hh] R_{mkql}[\hh]\)(y) = 0$,
\item[(3)] $\pi = 0$ on $M$, $R_{NN;N}[\bg](y) = R_{aN}[\bg](y) = 0$,
\item[(4)] $R_{;ii}[\bg](y) = -\dfrac{n\|W\|^2}{6(n-1)}$, $R_{NN;ii}[\bg](y) = - \dfrac{\|W\|^2}{12(n-1)}$, $R_{iNjN}[\bg](y) = R_{ij}[\bg](y)$,
\item[(5)] $R_{NN;NN}[\bg](y) = \dfrac{3}{2n} R_{;NN}[\bg](y) - 2 (R_{ij}[\bg](y))^2$,
\item[(6)] $R_{iNjN;ij}[\bg](y) = \(\dfrac{3-n}{2n}\) R_{;NN}[\bg](y) - (R_{ij}[\bg](y))^2 - \dfrac{\|W\|^2}{12(n-1)}$
\end{itemize}
if normal coordinates around $y \in (M,\hh)$ is assumed.
Here $\|W\|$ is the norm of the Weyl tensor of $(M,\hh)$ at $y$.
\end{lemma}
\noindent Note that the first partial derivatives of $\hh$ and the Christoffel symbols $\Gamma^k_{ij}[\hh] = \Gamma^k_{ij}[\bg]$ at $y$ vanish.
Also a simple computation utilizing $\pi = 0$ on $M$ shows that
$\Gamma^b_{aa}[\bg] = \Gamma^a_{bN}[\bg] = 0$ on $M$.
\begin{proof}[Proof of Lemma \ref{lemma_rep_2}]
\cite[Theorem 5.2]{LP} guarantees the existence of a representative $\hh \in [\hh]$ on $M$ such that (1), (2) and $R_{ij}[\hh](y) = 0$ hold.
Furthermore, \cite[Proposition 1.2]{Es2} shows that umbilicity is preserved under the conformal transformation, and so $\pi = 0$ on $M$.
The proof of the remaining identities in (3)-(6) is presented in 2 steps.

\medskip
\noindent \textsc{Step 1.} By differentiating \eqref{eq_R_gp} in $x_N$ and using the assumption that $\pa_N \(R[g^+] + n(n+1)\) = o(x_N^2)$ as $x_N \to 0$ (see \eqref{eq_main_l_2}), we obtain
\begin{equation}\label{eq_dif}
o(1) = n \left[{\pa_N |\bg| \over |\bg| x_N^2} + {\pa_{NN} |\bg| \over |\bg| x_N}
- {\(\pa_N |\bg|\)^2 \over |\bg|^2 x_N}\right] + {2 R[\bg] \over x_N} + R_{,N}[\bg] \quad \text{as } x_N \to 0.
\end{equation}
Also, since we supposed that the mean curvature $H$ vanishes on the umbilic boundary $M$, we get from \eqref{eq_rep} that $R_{NN}[\bg](y) = \pi(y) = 0$.
This in turn gives that $|\bg|(y) = 1$ and $\pa_N |\bg|(y) = \pa_{NN} |\bg|(y) = R[\bg](y) = 0$.
Consequently, by taking the limit to $y$ in \eqref{eq_dif}, we find that
\begin{equation}\label{eq_dif_2}
\begin{aligned}
0 &= n \left[{\pa_{NNN} |\bg|(y) \over 2} + {\pa_{NNN} |\bg|(y)} - 0 \right] + 2R_{,N}[\bg](y) + R_{,N}[\bg](y) \\
&= n \pa_{NNN} |\bg|(y) + 2R_{,N}[\bg](y).
\end{aligned}
\end{equation}
Now we observe from Lemma \ref{lemma_metric_2} that $\pa_{NNN} |\bg|(y) = -2R_{NN;N}[\bg](y)$.
In addition, by the second Bianchi identity, the Codazzi equation 
and the fact that $\pi = 0$ on $M$, one can achieve
\begin{equation}\label{eq_R_r}
\begin{aligned}
R_{,N}[\bg] &= R_{;N}[\bg] = 2R_{NN;N}[\bg] + R_{ijij;N}[\bg] = 2R_{NN;N}[\bg] + (R_{ijiN;j}[\bg] - R_{ijjN;i}[\bg])\\
&= 2R_{NN;N}[\bg] + 2(\pi_{ii;jj} - \pi_{ij;ij}) = 2R_{NN;N}[\bg]
\end{aligned}
\end{equation}
and
\[R_{iN}[\bg] = \pi_{jj;i} - \pi_{ij;j} = 0\]
at $y \in M$. Combining \eqref{eq_dif_2} and \eqref{eq_R_r}, we get
\[0 = (2-n) R_{NN;N}[\bg](y).\]
Since $n \ge 3$, it follows that $R_{NN;N}[\bg](y) = 0$ as we wanted.

\medskip
\noindent \textsc{Step 2.}
It is well-known that $R_{,ii}[\hh](y) = R_{;ii}[\hh](y) = -\|W(y)\|^2/6$ in the normal coordinate around $y \in M$.
Therefore the Gauss-Codazzi equation and the fact that $H = \pi = 0$ on $M$ imply
\begin{equation}\label{eq_rep_21}
R_{,ii}[\bg](y) = 2R_{NN,ii}[\bg](y) - {\|W(y)\|^2 \over 6} \quad \text{and} \quad R_{iNjN}[\bg](y) = R_{ij}[\bg](y).
\end{equation}
Moreover, since $\Delta_{\bx} \(R[g^+] + n(n+1)\) = o(x_N^2)$ near $y \in \ox$ (refer to \eqref{eq_main_l_2}),
by differentiating \eqref{eq_R_gp} in $x_i$ twice, dividing the result by $x_N^2$ and then taking the limit to $y$, one obtains
\begin{equation}\label{eq_rep_22}
R_{,ii}[\bg](y) = 2n R_{NN,ii}[\bg](y).
\end{equation}
As a result, putting \eqref{eq_rep_22} into \eqref{eq_rep_21} and applying the relations at $y$
\[R_{;ii}[\bg] = R_{,ii}[\bg] 
\quad \text{and} \quad R_{NN;ii}[\bg] = R_{NN,ii}[\bg] - 2 (\pa_i\Gamma^a_{iN}[\bg]) R_{aN}[\bg] \underset{\text{by }(3)}{=} R_{NN,ii}[\bg]\]
allow one to find (4).

On the other hand, arguing as before but using the hypothesis that $\pa_{NN} \(R[g^+] + n(n+1)\) = o(x_N^2)$ near $y \in \ox$ at this time, one derives equalities
\[3 R_{,NN}[\bg](y) = -n\, \pa_{NNNN} |\bg|(y) = 2n \left[ R_{NN,NN}[\bg](y) + 2(R_{iNjN}[\bg](y))^2 \right].\]
Because $R_{;NN}[\bg](y) = R_{,NN}[\bg](y)$ and $R_{NN;NN}[\bg](y) = R_{NN,NN}[\bg](y)$, it is identical to (5).
Hence the contracted second Bianchi identity, the Ricci identity and (3)-(5) give
\begin{align*}
R_{;NN}[\bg] &= 2R_{iN;iN}[\bg] + 2R_{NN;NN}[\bg] = 2\left[R_{iN;Ni}[\bg] + (R_{ij}[\bg])^2 - (R_{aN}[\bg])^2 \right] + 2R_{NN;NN}[\bg] \\
&= 2 \(R_{iN;Ni}[\bg] + (R_{ij}[\bg])^2\) + \({3 \over n} R_{;NN}[\bg] - 4 (R_{ij}[\bg])^2\).
\end{align*}
at $y$. Now assertion (6) directly follows from the above equality and
\[R_{iN;Ni}[\bg](y) = R_{Njij;Ni}[\bg](y) = - R_{iNjN;ij}[\bg](y) + R_{NN;ii}[\bg](y) = - R_{iNjN;ij}[\bg](y) - {\|W(y)\|^2 \over 12(n-1)}.\]
This finishes the proof.
\end{proof}

\subsection{Computation of the Energy} \label{subsec_um_ene}
Like the previous section, we fix a smooth radial cut-off function $\psi \in C^{\infty}_c(\mr_+^N)$ such that $\psi = 1$ in $B^N_+(0, \eta_0)$ and 0 in $\mr_+^N \setminus B^N_+(0, 2\eta_0)$.
Also, assume that $W_{\ep} = W_{\ep, 0}$ denotes the bubble defined in \eqref{eq_bubble}.
\begin{lemma}\label{lemma_h_2}
Let $y = 0 \in M$ be any fixed point and $J^{\ga}_{\hh}$ the functional given in \eqref{eq_J}.
If \eqref{eq_rep} and (1)-(6) in Lemma \ref{lemma_rep_2} are valid, then
\begin{equation}\label{eq_h_0}
\begin{aligned}
&\ J^{\ga}_{\hh}[\psi W_{\ep}; B^N_+(0,\eta_0)] \\
&= \int_{B^N_+(0, \eta_0)} x_N^{1-2\ga} |\nabla W_1|^2 dx
+ \ep^4 \left[ {\|W\|^2 \over 4n} \({\mcf'_5 \over 12(n-1)} - {\mcf'_6 \over 2(n-1)(n+2)} - {(n-2\ga)\mcf'_4 \over 12n}\) \right.\\
&\hspace{25pt} \left. + {R_{;NN}[\bg] \over 2} \(-{\mcf'_2 \over 8n} + {\mcf'_3 \over 4n^2} - {(n-3) \mcf'_6 \over n^2(n+2)} + {(n-2\ga)\mcf'_1 \over 4n} \)
+ {(R_{ij}[\bg])^2 \over n} \({\mcf'_3 \over 2} - {\mcf'_6 \over n+2}\) \right]\\
&\quad + o(\ep^4)
\end{aligned}
\end{equation}
for any $\ep > 0$ small, $n > 4 + 2\ga$ and $\ga \in (0,1)$.
Here the tensors are computed at $y$ and the values $\mcf'_1, \cdots, \mcf'_6$ are given in Lemma \ref{lemma_int_2}.
\end{lemma}
\begin{proof}
\textsc{Step 1 (Estimate on the second and third order terms).}
To begin with, we ascertain that
\begin{equation}\label{eq_h_1}
J^{\ga}_{\hh}[\psi W_{\ep}; B^N_+(0,\eta_0)] = \int_{B^N_+(0, \eta_0)} x_N^{1-2\ga} |\nabla W_1|^2 dx + O(\ep^4).
\end{equation}
In fact, since $H = R_{NN}[\bg] =0$ at $y$ and the bubbles $W_{\ep}$ depends only on the variables $|\bx|$ and $x_N$, we have
\begin{multline}\label{eq_h_2}
\int_{B^N_+(0, \eta_0)} x_N^{1-2\ga}|\nabla W_{\ep}|_{\bg}^2\, dv_{\bg} = \int_{B^N_+(0, \eta_0)} x_N^{1-2\ga} |\nabla W_1|^2 dx \\
+ \ep^3 R_{NN;N}[\bg](y) \({1 \over 3n} \int_{\mr^N_+} x_N^{4-2\ga} |\nabla_{\bx} W_1|^2 dx - {1 \over 6} \int_{\mr^N_+} x_N^{4-2\ga} |\nabla W_1|^2 dx\) + O(\ep^4).
\end{multline}
Moreover, thanks to \eqref{eq_main_l_2}, \eqref{eq_E_2} and $R[\bg](y) = R_{,N}[\bg](y) = 0$, it holds that
\begin{equation}\label{eq_h_3}
\begin{aligned}
&\ \int_{B^N_+(0, \eta_0)} E(x_N) W_{\ep}^2 \, dv_{\bg}\\
&= \int_{B^N_+(0, \eta_0)} E(x_N) W_{\ep}^2 dx + O\(\ep^{4+\zeta} \int_{B^N_+(0, \eta_0)} x_N^{1-2\ga} W_1^2 |x|^{4+\zeta} dx\)\\
&= \ep^2 \({n-2\ga \over 4n}\) \int_{B^N_+(0, \eta_0/\ep)} x_N^{1-2\ga} \(R[\bg](y) + \ep R_{,a}[\bg](y) x_a + {\ep^2 \over 2} R_{,ab}[\bg](y) x_ax_b \) W_1^2 dx + o(\ep^4)\\
&= \ep^4 \({n-2\ga \over 4n}\) \cdot \left[{1 \over 2n} R_{;ii}[\bg](y)\, \mcf'_4 + {1 \over 2} R_{;NN}[\bg](y)\, \mcf'_1 \right] + o(\ep^4)
\end{aligned}
\end{equation}
where $\zeta > 0$ is a sufficiently small number.
Because $R_{NN;N}[\bg](y) = 0$ by Lemma \ref{lemma_rep_2} (3), we see from \eqref{eq_h_2} and  \eqref{eq_h_3} that estimate \eqref{eq_h_1} is true.

\medskip
\noindent \textsc{Step 2 (Estimate on the fourth order terms).}
Let $\sqrt{|\bg|}^{(4)}$ and $(\bg^{ij})^{(4)}$ be the fourth order terms in the expansions \eqref{eq_bg_exp_4} and \eqref{eq_b_exp} of $\sqrt{|\bg|}$ and $\bg^{ij}$.
In view of \eqref{eq_rem}, Lemma \ref{lemma_rep_2} (2) and \cite[Corollary 29]{Br}, one can show that
\begin{multline*}
\int_{B^N_+(0, \eta_0)} x_N^{1-2\ga} |\nabla W_\ep|^2 \sqrt{|\bg|}^{(4)} dx \\
= - \ep^4 \left[{1 \over 4n} R_{NN;ii}[\bg](y)\, \mcf'_5 + {1 \over 24} \( R_{NN;NN}[\bg](y) + 2 (R_{iNjN}[\bg](y))^2 \) \mcf'_2 \right] + o(\ep^4)
\end{multline*}
and
\begin{multline*}
\int_{B^N_+(0, \eta_0)} x_N^{1-2\ga} (\bg^{ij})^{(4)} \pa_iW_\ep \pa_jW_\ep dx
= \ep^4 \left[ {1 \over 2n(n+2)} \( R_{NN;ii}[\bg](y) + 2R_{iNjN;ij}[\bg](y) \) \mcf'_6 \right.
\\
\left. + {1 \over 12n} \( R_{NN;NN}[\bg](y) + 8 (R_{iNjN}[\bg](y))^2 \) \mcf'_3 \right] + o(\ep^4)
\end{multline*}
(cf. \cite[Section 4]{GW}).
Therefore \eqref{eq_rep}, \eqref{eq_h_2} and Lemma \ref{lemma_rep_2} (4)-(6) yield
\begin{align*}
&\ \int_{B^N_+(0, \eta_0)} x_N^{1-2\ga}|\nabla W_{\ep}|_{\bg}^2\, dv_{\bg} \\
&= \int_{B^N_+(0, \eta_0)} x_N^{1-2\ga} |\nabla W_1|^2 dx
+ \ep^4 \left[ {\|W\|^2 \over 8n(n-1)} \({\mcf'_5 \over 6} - {\mcf'_6 \over n+2} \) \right.\\
&\hspace{100pt} \left. + {R_{;NN}[\bg] \over 2n} \(-{\mcf'_2 \over 8} + {\mcf'_3 \over 4n} - {(n-3) \mcf'_6 \over n(n+2)} \)
+ {(R_{ij}[\bg])^2 \over n} \({\mcf'_3 \over 2} - {\mcf'_6 \over n+2}\) \right] + o(\ep^4).
\end{align*}
Now \eqref{eq_h_3} and the previous estimate lead us to \eqref{eq_h_0}.
The proof is accomplished.
\end{proof}

\begin{cor}
Assume that $(X^{n+1},g^+)$ is an asymptotically hyperbolic manifold,
$\hh$ is the representative of the conformal infinity $M$ in Lemma \ref{lemma_metric_2}
and $\overline{I}_{\hh}^{\ga}$ is the $\ga$-Yamabe functional in \eqref{eq_yamabe_q}.
If $n > 4 + 2\ga$ and $\ga \in (0,1)$, we have that
\begin{multline}\label{eq_um_0}
\overline{I}_{\hh}^{\ga}[\psi W_{\ep}] \le \Lambda^{\ga}(\ms^n, [g_c]) + \ep^4 \Lambda^{\ga}(\ms^n, [g_c])^{-{n-2\ga \over 2\ga}} \kappa_{\ga}\, |\ms^{n-1}| A_3B_2\\
\times \(- \|W\|^2 \mcd'_1(n, \ga)\, + R_{;NN}[\bg] \mcd'_2(n, \ga) - (R_{ij}[\bg])^2 \mcd'_3(n, \ga)\) + o(\ep^4)
\end{multline}
where the positive constants $\Lambda^{\ga}(\ms^n, [g_c])$, $\kappa_{\ga}$, $A_3$ and $B_2$ are introduced in \eqref{eq_Lambda}, \eqref{eq_const} and \eqref{eq_AB}, respectively. Furthermore
\begin{equation}\label{eq_mcd'}
\begin{aligned}
\mcd'_1(n, \ga) &= \tfrac{15n^4 - 135n^3 + 10n^2(43+3\ga-4\ga^2) - 180n(3+\ga-\ga^2) + 8(24+35-30\ga^2-5\ga^3+6\ga^4)} {480n(n-1)(n-4)(n-4-2\ga)(n-4+2\ga)(1-\ga^2)} > 0,\\
\mcd'_2(n, \ga) &= 0
\end{aligned}
\end{equation}
and
\[\mcd'_3(n, \ga) = {5n^2 - 4n(13-2\ga^2) + 28(4-\ga^2) \over 5n(n-4)(n-4-2\ga)(n-4+2\ga)}.\]
\end{cor}
\begin{proof}
By Lemmas \ref{lemma_metric_2} and \ref{lemma_rep_2} (1)-(2), it holds that
\begin{align*}
&\ \int_M (\psi W_{\ep})^{p+1} dv_{\hh} \\
&= \int_{B^n(0,\eta_0)} w_{\ep}^{p+1} \left[1 - {1 \over 40} \(R_{ij,kl}[\hh] + {2 \over 9} R_{miqj}[\hh] R_{mkql}[\hh]\) x_ix_jx_kx_l + O(|\bx|^5) \right] d\bx + O(\ep^n) \\
&= \int_{\mr^n} w_1^{p+1} d\bx + o(\ep^4).
\end{align*}
Thus the conclusion follows from Lemmas \ref{lemma_h_2} and \ref{lemma_int_2} at once.
\end{proof}
\noindent \noindent It is interesting to observe that the quantity $R_{;NN}[\bg](y)$ does not contribute to the existence of a least energy solution,
since the coefficient of $R_{;NN}[\bg](y)$, denoted by $\mcd'_2(n, \ga)$, is always zero for any $n$ and $\gamma$.
Such a phenomenon has been already observed in the boundary Yamabe problem \cite{Ma}.
We also observe that the number $\mcd'_3(n, \ga)$ has a nonnegative sign in some situations: when $n = 7$ and $\ga \in [1/2, 1)$, or $n \ge 8$ and $\ga \in (0,1)$.
In order to cover lower dimensional cases, we need a more refined test function.

\medskip
Let $y \in M$ be a point such that $W[\hh](y) \ne 0$. Motivated by \cite{Al}, we define functions
\[\wPsi_{\ep} = \Psi_{\ep}(\bx, x_N) = M_2 R_{iNjN}[\bg] x_ix_j x_N^2 r^{-1} \pa_r W_{\ep} = \ep^2 \cdot \ep^{-{n-2\ga \over 2}} \wPsi_1(\ep^{-1}\bx, \ep^{-1} x_N)\]
for some $M_2 \in \mr$ and
\[\wPhi_{\ep} := \psi(W_{\ep} + \wPsi_{\ep}) \quad \text{in } X.\]
\begin{prop}
Suppose that $(X^{n+1},g^+)$ is an asymptotically hyperbolic manifold.
Moreover $\hh$ is the representative of the conformal infinity $M$ satisfying \eqref{eq_rep} and Lemma \ref{lemma_rep_2} (1)-(6).
If $n > 4 + 2\ga$ and $\ga \in (0,1)$, we have
\begin{multline}\label{eq_um_1}
\overline{I}_{\hh}^{\ga}[\wPhi_{\ep}] \le \Lambda^{\ga}(\ms^n, [g_c]) + \ep^4 \Lambda^{\ga}(\ms^n, [g_c])^{-{n-2\ga \over 2\ga}} \kappa_{\ga}\, |\ms^{n-1}| A_3B_2\\
\times \(- \|W\|^2 \mcd_1(n, \ga)\, + R_{;NN}[\bg] \mcd_2(n, \ga) - (R_{ij}[\bg])^2 \mcd_3(n, \ga)\) + o(\ep^4)
\end{multline}
where
\[\mcd_1(n, \ga) = \mcd'_1(n, \ga),\quad \mcd_2(n, \ga) = 0\]
(see \eqref{eq_mcd'} for the definition of the positive constant $\mcd'_1(n, \ga)$) and
\[\mcd_3(n, \ga) = \frac{25n^3 - 20n^2(9-\ga^2) + 100n(4-\ga^2) - 16(4 - \ga^2)^2}{5n (n-4-2\ga)(n-4+2\ga) (5n^2 - 4n(1+\ga^2) - 8 (4-\ga^2))}.\]
\end{prop}
\begin{proof}
Since $R_{NN}[\bg](y) = 0$, we obtain
\begin{equation}\label{eq_um_2}
\begin{aligned}
J^{\ga}_{\hh}\left[\wPhi_{\ep}; B^N_+(0,\eta_0)\right] &= J^{\ga}_{\hh} \left[\psi W_{\ep}; B^N_+(0, \eta_0)\right]
+ 2 \int_{\mr^N_+} x_N^{1-2\ga} (\bg^{ij} - \delta^{ij})\, \pa_i W_{\ep} \pa_j \wPsi_{\ep} \, dx\\
&\ + \int_{\mr^N_+} x_N^{1-2\ga} |\nabla \wPsi_{\ep}|^2 dx + o(\ep^4).
\end{aligned}
\end{equation}
Also a tedious computation with Lemmas \ref{lemma_metric_2} and \ref{lemma_rep_2} (4) reveals that the second term of the right-hand side of \eqref{eq_um_2} is equal to
\begin{multline*}
{2 \over 3} R_{ikjl}[\hh] \int_{\mr^N_+} x_N^{1-2\ga} x_kx_l\, \pa_i W_{\ep} \pa_j \wPsi_{\ep} \, dx
+ 2R_{iNjN}[\bg] \int_{\mr^N_+} x_N^{3-2\ga} \pa_i W_{\ep} \pa_j \wPsi_{\ep} \, dx + o(\ep^4) \\
= 0 + \ep^4 4M_2 \left[{1 \over n} \mcf'_3 + {1 \over n(n+2)} (-\mcf'_3 + \mcf'_7) \right] (R_{ij}[\bg])^2 + o(\ep^4)
\end{multline*}
and it holds that
\[\int_{\mr^N_+} x_N^{1-2\ga} |\nabla \wPsi_{\ep}|^2 dx = \ep^4 \left[{2M_2^2 \over n(n+2)}\right] \(\mcf'_3 - 2\mcf'_7 + \mcf'_8 + 4\mcf'_6 + 4\mcf'_9 + \mcf'_{10}\) (R_{ij}[\bg])^2 + o(\ep^4) \]
(cf. \eqref{eq_l_6} and \eqref{eq_l_7}).
Here the constants $\mcf'_1, \cdots, \mcf'_{10}$ are defined in Lemma \ref{lemma_int_2}.

On the other hand, we have
\[J^{\ga}_{\hh}\left[\wPhi_{\ep}; X \setminus B^N_+(0,\eta_0) \right] = o(\ep^4),\]
and since $\wPsi_{\ep} = 0$ on $M$, the integral of $|\wPhi_{\ep}|^{p+1}$ over the boundary $M$ does not contribute to the fourth order term in the right-hand side of \eqref{eq_um_1}.
By combining all information, employing Lemma \ref{lemma_int_2} and selecting the optimal $M_2 \in \mr$, we complete the proof.
\end{proof}
\noindent One can verify that $\mcd_3(n, \ga) > 0$ whenever $n > 4 + 2\ga$ and $\ga \in (0,1)$.
Consequently we deduce Theorem \ref{thm_main_l_2} from the previous proposition.

\section{Locally Conformally Flat or 2-dimensional Conformal Infinities} \label{sec_lcf}
\subsection{Analysis of Green's function} \label{subsec_Green}
In this subsection, we prove Proposition \ref{prop_G_exp}.
By Theorem \ref{thm_ext}, solvability of problem \eqref{eq_G} for each $y \in M$ is equivalent to the existence of a solution $G^*$ to the equation
\[\begin{cases}
-\text{div}_{\bg^*}\((\rho^*)^{1-2\ga}\nabla G^*(\cdot,y)\) = 0 &\text{in } (X, \bg^*),\\
\pa^{\ga}_{\nu} G^*(\cdot,y) = \delta_y - Q^{\ga}_{\hh} G^*(\cdot, y) &\text{on } (M, \hh),
\end{cases}\]
and it holds that $|\bg^*_{iN}| + |\bg^*_{NN} - 1| = O(\rho^{2\ga})$.
We also recall \cite[Corollary 4.3]{GQ} which states that if $\Lambda^{\ga}(M, [\hh]) > 0$, then $M$ admits a metric $\hh_0 \in [\hh]$ such that $Q^{\ga}_{\hh_0} > 0$ on $M$.
Thanks to the following lemma, it suffices to show Proposition \ref{prop_G_exp} for $\hh_0 \in [\hh]$.
\begin{lemma}
Let $(X, g^+)$ be any conformally compact Einstein manifold with conformal infinity $(M, [\hh])$, $\rho$ the geodesic defining function of $M$ in $X$ and $\bg = \rho^2 g^+$.
For any positive smooth function $w$ on $M$, define a new metric $\hh_w = w^{4 \over n-2\ga} \hh$, denote the corresponding geodesic boundary defining function by $\rho_w$ and set $\bg_w = \rho_w^2 g^+$.
Suppose that $G = G(x,y)$ solves \eqref{eq_G}. Then the function
\[G_w(x,y) := \({\rho(x) \over \rho_w(x)}\)^{n-2\ga \over 2} w^{n+2\ga \over n-2\ga}(y)\, G(x,y) \quad \text{for } (x,y) \in \ox \times M,\, x \ne y\]
again satisfies \eqref{eq_G} with $(\bg_w, \hh_w)$ and $\rho_w$ substituted for $(\bg, \hh)$ and $\rho$, respectively.
\end{lemma}
\begin{proof}
By \eqref{eq_E}, the first equality in \eqref{eq_G} is re-expressed as
\begin{equation}\label{eq_G_0}
L_{\bg}\(\rho^{1-2\ga \over 2} G(\cdot, y)\) + \(\ga^2 - {1 \over 4}\) \rho^{-\({3+2\ga \over 2}\)} G(\cdot,y) = 0 \quad \text{in } (X, \bg)
\end{equation}
where $L_{\bg}$ is the conformal Laplacian in $(X, \bg)$ defined in \eqref{eq_conf_L}.
Therefore one observes from \eqref{eq_conf_cov} that $G_w$ is a solution of \eqref{eq_G_0} if $\bg$ and $\rho$ are replaced with $\bg_w$ and $\rho_w$, respectively.
Also, since $w = (\rho_w/\rho)^{(n-2\ga)/2}$ on $M$, we see
\begin{align*}
\pa^{\ga}_{\nu} G_w(\cdot,y) &= P^{\ga}_{\hh_w} G_w(\cdot,y)
= w^{n+2\ga \over n-2\ga}(y)\, P^{\ga}_{w^{4 \over n-2\ga} \hh}\( (\rho/\rho_w)^{n-2\ga \over 2} \, G(\cdot,y)\) \\
&= w^{n+2\ga \over n-2\ga}(y)\, P^{\ga}_{w^{4 \over n-2\ga} \hh}\( w^{-1} \, G(\cdot,y)\)
= w^{n+2\ga \over n-2\ga}(y)\, w^{-{n+2\ga \over n-2\ga}} P^{\ga}_{\hh}(G(\cdot,y)) \\
& = w^{n+2\ga \over n-2\ga}(y)\, w^{-{n+2\ga \over n-2\ga}} \pa^{\ga}_{\nu}(G(\cdot,y)) = w^{n+2\ga \over n-2\ga}(y)\, w^{-{n+2\ga \over n-2\ga}} \delta_y = \delta_y \quad \text{on } M
\end{align*}
where we have applied Theorem \ref{thm_ext} and \eqref{eq_conf_cov} for the first, fourth and fifth equalities.
\end{proof}

For brevity, we write $\hh = \hh_0$, $\bg = \bg^*$, $\rho = \rho^*$ and $G = G^*$ here and henceforth.
Further, recalling that $Q^{\ga}_{\hh} > 0$ on $M$, let us define a norm
\[\|U\|_{\mcw^{1,q}\xr} = \(\int_X \rho^{1-2\ga} |\nabla U|_{\bg}^q \, dv_{\bg} + \int_M Q^{\ga}_{\hh} U^q dv_{\hh}\)^{1/q}\]
for any $q \ge 1$ and set a space $\mcw^{1,q}\xr$ as the completion of $C^{\infty}_c(\ox)$ with respect to the above norm.

Given any bounded Radon measure $f$ (such as the dirac measures), we say that a function $U \in \mcw^{1,q}\xr$ is a {\it weak solution} of
\begin{equation}\label{eq_G_2}
\begin{cases}
-\text{div}_{\bg}\(\rho^{1-2\ga}\nabla U\) = 0 &\text{in } (X, \bg),\\
\pa^{\ga}_{\nu} U + Q^{\ga}_{\hh} U = f  &\text{on } (M, \hh),
\end{cases}
\end{equation}
if it is satisfied that
\begin{equation}\label{eq_weak}
\int_X \rho^{1-2\ga} \la \nabla U, \nabla \Psi \ra_{\bg} dv_{\bg} + \int_M Q^{\ga}_{\hh} U \Psi dv_{\hh} = \int_M f \Psi
\end{equation}
for any $\Psi \in C^1(\ox)$.

The $\mcw^{1,2}\xr$-norm is equivalent to the standard weighted Sobolev norm $\|U\|_{W^{1,2}\xr}$ (see \cite[Lemma 3.1]{CK}).
Thus for any fixed $f \in (H^{\ga}(M))^*$, the existence and uniqueness of a solution $U \in W^{1,2}\xr$ to \eqref{eq_G_2} are guaranteed by the Riesz representation theorem.
\begin{lemma}\label{lemma_weak_reg}
Assume that $n > 2\ga$, $f \in (H^{\ga}(M))^*$ and $1 \le \alpha < \min\{{n \over n-2\ga}, {2n+2 \over 2n+1}\}$.
Then there exists a constant $C = C(\ox, g^+, \rho, n, \ga, \alpha)$ such that
\begin{equation}\label{eq_weak_reg}
\|U\|_{\mcw^{1,\alpha}\xr} \le C\|f\|_{L^1(M)}
\end{equation}
for a weak solution $U \in W^{1,2}\xr$ to \eqref{eq_G_2}.
As a result, if $f$ is the dirac measure $\delta_y$ at $y \in M$, then \eqref{eq_G_2} has a unique nonnegative weak solution $G(\cdot,y) \in \mcw^{1,\alpha}\xr$.
\end{lemma}
\begin{proof}
\textsc{Step 1.} We are going to verify estimate \eqref{eq_weak_reg} by suitably modifying the argument in \cite[Section 5]{BS}.
To this aim, we consider the {\it formal adjoint} of \eqref{eq_G_2}:
Given any $h_0 \in L^q(M)$ and $H_1, \cdots, H_N \in L^q\xr$ for some $q > \max\{{n \over 2\ga}, 2(n+1)\}$, we study a function $V$ such that
\begin{equation}\label{eq_w_dual}
\int_X \rho^{1-2\ga} \la \nabla V, \nabla \Psi \ra_{\bg} dv_{\bg} + \int_M Q^{\ga}_{\hh} V \Psi dv_{\hh} = \int_M h_0 \Psi dv_{\hh} + \sum_{a=1}^N \int_X \rho^{1-2\ga} H_a \pa_a \Psi dv_{\bg}
\end{equation}
for any $\Psi \in C^1(\ox)$.
Indeed, by the Lax-Milgram theorem, \eqref{eq_w_dual} possesses a unique solution $V \in W^{1,2}\xr$.
Moreover, as will be seen in Step 3 below, it turns out that $V$ satisfies
\begin{equation}\label{eq_w_dual_est}
\|V\|_{L^{\infty}(M)} + \|V\|_{L^{\infty}(X)} \le C \(\|h_0\|_{L^q(M)} + \sum_{a=1}^N \|H_a\|_{L^q\xr}\).
\end{equation}
Therefore taking $\Psi = U$ in \eqref{eq_weak} (which is allowed to do thanks to the density argument) and employing \eqref{eq_w_dual_est}, we find
\begin{align*}
\int_M U h_0 dv_{\hh} + \sum_{a=1}^N \int_X \rho^{1-2\ga} \pa_a U H_a dv_{\bg} &= \int_M fV dv_{\hh} \le \|f\|_{L^1(M)} \|V\|_{L^{\infty}(M)} \\
&\le C \|f\|_{L^1(M)} \(\|h_0\|_{L^q(M)} + \sum_{a=1}^N \|H_a\|_{L^q\xr}\).
\end{align*}
This implies the validity of \eqref{eq_weak_reg} with $\alpha = q'$ where $q'$ designates the H\"older conjugate of $q$.

\medskip \noindent \textsc{Step 2.}
Assume now that $f = \delta_y$ for some $y \in M$.
Then one is capable of constructing a sequence $\{f_m\}_{m \in \mn} \subset C^1(M)$ with an approximation to the identity or a mollifier so that $f_m \ge 0$ on $M$,
\[\sup_{m \in \mn} \|f_m\|_{L^1(M)} \le C, \quad f_m \to 0 \text{ in } C^1_{\text{loc}}(M \setminus \{y\}) \quad \text{and} \quad
f_m \rightharpoonup \delta_y \text{ in the distributional sense.}\]
Denote by $\{U_m\}_{m \in \mn} \subset W^{1,2}\xr$ a sequence of the corresponding weak solutions to \eqref{eq_G_2}.
By \eqref{eq_weak_reg} and elliptic regularity, there exist a function $G(\cdot, y)$ and a number $\ve_0 \in (0,1)$
such that $U_m \rightharpoonup G(\cdot, y)$ weakly in $\mcw^{1,\alpha}\xr$ and $U_m \to G(\cdot, y)$ in $C^{\ve_0}_{\text{loc}}(\ox \setminus \{y\})$.
It is a simple task to confirm that $G(\cdot, y)$ satisfies \eqref{eq_weak}.

Also, putting $(U_m)_- \in W^{1,2}\xr$ into \eqref{eq_weak} yields $U_m \ge 0$ in $X$,
which in turn gives $G(\cdot,y) \ge 0$ in $X$.
Finally, it is easy to see that the uniqueness of $G(\cdot,y)$ comes as a consequence of \eqref{eq_weak_reg}.
This completes the proof of the lemma except \eqref{eq_w_dual_est}.

\medskip \noindent \textsc{Step 3 (Justification of estimate \eqref{eq_w_dual_est}).}
We shall apply Moser's iteration technique so as to get \eqref{eq_w_dual_est}. Set
\[\zeta_0 = \|h_0\|_{L^q(M)} + \sum_{a=1}^N \|H_a\|_{L^q\xr} \quad \text{if } (h_0, H_1, \cdots, H_N) \ne (0, 0, \cdots, 0).\]
Otherwise let $\zeta_0$ be any positive number which we will make $\zeta_0 \to 0$ eventually. Then we define $\ov = V_+ + \zeta_0$ and
\[\ov_\ell = \begin{cases}
\ov &\text{if } V < \ell,\\
\ell + \zeta_0 &\text{if } V \ge \ell
\end{cases}\]
for each $\ell > 0$. Testing
\[\Psi = \ov_\ell^{\beta-1} \ov - \zeta_0^{\beta} \in W^{1,2}\xr\]
in \eqref{eq_w_dual} for a fixed exponent $\beta \ge 1$ shows that
\begin{multline} \label{eq_w_dual_1}
{1 \over 2\beta} \(\int_X \rho^{1-2\ga} \left| \nabla \wtv_{\ell} \right|_{\bg}^2 dv_{\bg}
+ \int_M Q^{\ga}_{\hh} \wtv_{\ell}^2 dv_{\hh}\) \\
\le 2 \int_M Q^{\ga}_{\hh} \wtv_{\ell}^2 dv_{\hh} + {1 \over \zeta_0} \int_M |h_0| \wtv_{\ell}^2  dv_{\hh}
+ \int_X \rho^{1-2\ga} \(\sum_{a=1}^N |H_a|\) \left| \nabla \(\ov_\ell^{\beta-1} \ov\) \right|_{\bg} dv_{\bg}
\end{multline}
where $\wtv_{\ell} := \ov_{\ell}^{\beta-1 \over 2} \ov$.
Then one sees that \eqref{eq_w_dual_1} is reduced to
\begin{multline} \label{eq_w_dual_11}
{1 \over 4\beta} \(\int_X \rho^{1-2\ga} \left| \nabla \wtv_{\ell} \right|_{\bg}^2 dv_{\bg}
+ \int_M Q^{\ga}_{\hh} \wtv_{\ell}^2 dv_{\hh}\) \\
\le 2 \int_M Q^{\ga}_{\hh} \wtv_{\ell}^2 dv_{\hh} + {1 \over \zeta_0} \int_M |h_0| \wtv_{\ell}^2 dv_{\hh}
+ {4 \beta^2 \over \zeta_0^2} \int_X \rho^{1-2\ga} \(\sum_{a=1}^N |H_a|^2\) \wtv_{\ell}^2 dv_{\bg}.
\end{multline}
Besides an application of the Sobolev inequality and the Sobolev trace inequality (see \cite{FKS, Xi}) yields
\begin{equation}\label{eq_w_dual_2}
\(\int_M \wtv_{\ell}^{p+1} dv_{\hh}\)^{2 \over p+1} + \(\int_X \rho^{1-2\ga} \wtv_{\ell}^{2(n+1) \over n} dv_{\bg}\)^{n \over n+1}
\le C \left[ \int_X \rho^{1-2\ga} \left|\nabla \wtv_{\ell} \right|_{\bg}^2 dv_{\bg} + \int_M Q^{\ga}_{\hh} \wtv_{\ell}^2 dv_{\hh} \right],
\end{equation}
while H\"older's inequality gives
\begin{equation}\label{eq_w_dual_3}
\begin{aligned}
&\ {1 \over \zeta_0} \int_M |h_0| \wtv_{\ell}^2 dv_{\hh}
+ {4 \beta^2 \over \zeta_0^2} \int_X \rho^{1-2\ga} \(\sum_{a=1}^N |H_a|^2\) \wtv_{\ell}^2 dv_{\bg} \\
&\le \left[\delta_1^{1 \over \theta_1} \(\int_M \wtv_{\ell}^{p+1} dv_{\hh}\)^{2 \over p+1}
+ \delta_1^{-{1 \over 1-\theta_1}} \(\int_M \wtv_{\ell}^2 dv_{\hh}\) \right] \\
&\quad + 4 \beta^2 \left[\delta_2^{1 \over \theta_2} \(\int_X \rho^{1-2\ga} \wtv_{\ell}^{2(n+1) \over n} dv_{\bg}\)^{n \over n+1}
+ \delta_2^{-{1 \over 1-\theta_2}} \(\int_X \rho^{1-2\ga} \wtv_{\ell}^2 dv_{\bg}\) \right]
\end{aligned}
\end{equation}
for any small $\delta_1, \delta_2 > 0$ and some $\theta_1, \theta_2 \in (0,1)$ satisfying
\[{2\theta_1 \over p+1} + (1-\theta_1) = {n\theta_2 \over n+1} + (1-\theta_2) = {1 \over q'}.\]
Note that such numbers $\theta_1$ and $\theta_2$ exist because of the assumption that $q > \max\{{n \over 2\ga}, 2(n+1)\}$.
Collecting \eqref{eq_w_dual_11}-\eqref{eq_w_dual_3} and applying Lebesgue's monotone convergence theorem, we arrive at
\begin{multline*}
\(\int_M \ov^{(\beta+1) \cdot \({p+1 \over 2}\)} dv_{\hh}\)^{2 \over p+1} + \(\int_X \rho^{1-2\ga} \ov^{(\beta+1) \cdot \(n+1 \over n\)} dv_{\bg}\)^{n \over n+1} \\
\le C \beta \left[ \(2 + \beta^{\theta_1 \over 1-\theta_1}\) \(\int_M \ov^{\beta+1} dv_{\hh}\)
+ \beta^{3\theta_2 \over 1-\theta_2} \(\int_X \rho^{1-2\ga} \ov^{\beta+1} dv_{\bg}\) \right]
\end{multline*}
for a constant $C > 0$ independent of the choice of $\beta$.
Consequently, the standard iteration argument (considering also the replacement of $V$ with $-V$) reveals that
there exists $C > 0$ depending only on $\ox,\, g^+,\, \rho,\, n,\, \ga,\, \tal$ and $q$ for each $\tal \ge 2$ such that
\begin{equation}\label{eq_w_dual_4}
\|V\|_{L^{\infty}(M)} + \|V\|_{L^{\infty}(X)} \le C \( \|V\|_{L^{\tal}(M)} + \|V\|_{L^{\tal}\xr} + \|h_0\|_{L^q(M)} + \sum_{a=1}^N \|H_a\|_{L^q\xr} \).
\end{equation}
Now \eqref{eq_w_dual_est} is achieved in view of the compactness of the Sobolev embedding $W^{1,2}\xr \hookrightarrow L^2\xr$ (refer to \cite{GO} and \cite[Corollary A.1]{JX}),
that of the trace operator $W^{1,2}\xr \hookrightarrow L^{p+1-\ve_1}(M)$ for any small $\ve_1 > 0$,
the coercivity of the bilinear form in the left-hand side of \eqref{eq_w_dual} and the assumption $q > 2(n+1) \ge 2$.
\end{proof}

\begin{proof}[Completion of the proof of Proposition \ref{prop_G_exp}]
The existence and nonnegativity of Green's function $G$ is deduced in the previous lemma.
Owing to Hopf's lemma (cf. \cite[Theorem 3.5]{GQ}), $G$ is positive on the compact manifold $\ox$.
Remind that the coercivity of \eqref{eq_weak} implies the uniqueness of $G$. The proof is finished.
\end{proof}

\subsection{Locally Conformally Flat Case} \label{subsec_lcf}
This subsection is devoted to provide the proof of Theorem \ref{thm_main} under the hypothesis that $M$ is locally conformally flat.
Since the explicit solutions are known when $(X^{n+1},\bg) \simeq \mb^{n+1}$, we shall exclude such a case throughout the section.

\medskip
Pick any point $y \in M$. Since it is supposed to be locally conformally flat, we can assume that $y$ is the origin in $\mr^N$
and identify a neighborhood $\mathcal{U}$ of $y$ in $M$ with a Euclidean ball $B^n(0,\vr_1)$ for some $\vr_1 > 0$ small (namely, $\hh_{ij} = \delta_{ij}$ in $\mathcal{U} = B^n(0,\vr_1)$).
Write $x_N$ to denote the geodesic defining function $\rho$ for the boundary $M$ near $y$. Then we have smooth symmetric $n$-tensors $h^{(1)},\, \cdots,\, h^{(n-1)}$ on $B^n(0, \vr_1)$ such that
\begin{equation}\label{eq_bg_exp}
\bg = h_{x_N} \oplus dx_N^2 \quad \text{where } (h_{x_N})_{ij}(\bx, x_N) = \delta_{ij} + \sum_{m=1}^{n-1} h^{(m)}_{ij}(\bx) x_N^m + O(x_N^n)
\end{equation}
for $(\bx, x_N) \in \mcr^N(\vr_1, \vr_2) := B^n(0,\vr_1) \times [0,\vr_2) \subset \ox$ where $\vr_2 > 0$ is a number small enough.
In fact, as we will see shortly, the local conformal flatness on $M$ and the assumption that $X$ is Poincar\'e-Einstein
together imply that all low-order tensors $h^{(m)}$ which can be locally determined should vanish.
\begin{lemma}\label{lemma_h^m}
If $(X, g^+)$ is Poincar\'e-Einstein, we have $h^{(m)} = 0$ in \eqref{eq_bg_exp} for each $m = 1, \cdots, n-1$.
\end{lemma}
\begin{proof}
We adapt the idea in \cite[Lemma 7.7]{GQ} and \cite[Lemma 2.2]{GW}.
According to (2.5) of \cite{Gr}, it holds that
\begin{multline}\label{eq_h^m}
x_N h_{ij,NN} + (1-n) h_{ij,N} - h^{kl} h_{kl,N} h_{ij} - x_N h^{kl} h_{ik,N} h_{jl,N} + {1 \over 2} x_N h^{kl} h_{kl,N} h_{ij, N} - 2x_N R_{ij}[h]\\
= -2x_N (R_{ij}[g^+] + ng^+) = 0
\end{multline}
for $h := h_{x_N}$. Here the first equality is true for any metric $\bg$ satisfying \eqref{eq_bg_exp},
whereas the second equality holds because $(X, g^+)$ is Poincar\'e-Einstein.
Putting $x_N = 0$ in \eqref{eq_h^m}, we get
\[(1-n)\, h_{ij,N} - \hh^{kl} h_{kl,N} \hh_{ij} = 0,\]
from which we observe
\[\quad (1-n)\, \text{tr}_{\hh} h_{,N} - n\, \text{tr}_{\hh} h_{,N} = (1-2n)\, \text{tr}_{\hh} h_{,N} = 0.\]
It follows that the trace $\text{tr}_{\hh} h_{,N}$ is 0, and eventually, one finds $h_{,N} = h^{(1)} = 0$ on $\{x_N = 0\}$.

On the other hand, it holds that $R_{ij}[\hh] = 0$ on $\{x_N = 0\}$, for $\hh_{ij} = \delta_{ij}$.
Thus, by differentiating the both sides of \eqref{eq_h^m} in $x_N$ and taking $x_N = 0$, we obtain
\[(2-n)\, h_{ij,NN} - \hh^{kl} h_{kl,NN} \hh_{ij} = 0 \quad \text{and} \quad (2-2n)\, \text{tr}_{\hh} h_{,NN} = 0,\]
which again gives $h_{ij,NN} = h^{(2)} = 0$ on $\{x_N = 0\}$.

Analogously, if we differentiating \eqref{eq_h^m} $(m-1)$-times ($m = 3, \cdots, n$) and putting $x_N = 0$, then we have
\[(m-n)\, \pa_N^m h_{ij} - \hh^{kl} \(\pa_N^m h_{kl}\) \hh_{ij} = 0.\]
This gives $\pa_N^m h_{ij} = h^{(m)} = 0$ on $\{x_N = 0\}$, proving the lemma.
\end{proof}
\noindent In particular, the second fundamental form $h^{(1)}$ on $M$ (up to a constant factor) is 0, which indicates Remark \ref{rmk_sec_fund} (2).

\medskip Therefore it reduces to
\begin{equation}\label{eq_bg_exp_2}
\bg_{ij}(\bx,x_N) = \delta_{ij} + O(x_N^n) \quad \text{and} \quad
|\bg| = 1 + O(x_N^n)
\quad \text{for } (\bx, x_N) \in \mcr^N(\vr_1, \vr_2) \subset \ox.
\end{equation}
Now Conjecture \ref{conj_pos} implies that there is a solution $G(\cdot, 0)$ to \eqref{eq_G} with $y = 0$ such that
\[G(x,0) = g_{n,\ga} |x|^{-(n-2\ga)} + A + \Psi(x) \quad \text{for } x \in \mcr^N(\vr_1/2, \vr_2/2)\]
where $g_{n,\ga},\, A > 0$ are fixed constants and $\Psi$ is a function having the behavior \eqref{eq_Psi}.

Choose any smooth cut-off function $\chi: [0, \infty) \to [0,1]$ such that $\chi(t) = 1$ for $0 \le t \le 1$ and 0 for $t \ge 2$.
Then we construct a nonnegative, continuous and piecewise smooth function $\Phi_{\ep, \vr_0}$ on $\ox$ by
\begin{equation}\label{eq_Phi}
\Phi_{\ep, \vr_0}(x) = \begin{cases}
W_{\ep}(x) &\text{if } x \in X \cap B^N(0,\vr_0),\\
V_{\ep, \vr_0}(x) \, \(G(x,0)-\chi_{\vr_0}(x)\Psi(x)\) &\text{if } x \in X \cap \(B^N(0,2\vr_0) \setminus B^N(0,\vr_0)\),\\
V_{\ep, \vr_0}(x) \, G(x,0) &\text{if } x \in X \setminus B^N(0,2\vr_0)
\end{cases} \end{equation}
where $0 < \ep \ll \vr_0 \le \min\{\vr_1, \vr_2\}/5$ sufficiently small, $\chi_{\vr_0}(x) := \chi(|x|/\vr_0)$ and
\begin{equation}\label{eq_V}
V_{\ep, \vr_0}(x) := \left[\alpha_{n,\ga} \({\ep^{n-2\ga \over 2} \over \vr_0^{n-2\ga}}\)
+ \chi_{\vr_0}(x) \(W_{\ep}(x) - \alpha_{n,\ga} {\ep^{n-2\ga \over 2} \over |x|^{n-2\ga}}\)\right] \cdot \(\vr_0^{-(n-2\ga)} + A\)^{-1}.
\end{equation}
We remark that the main block $V_{\ep, \vr_0}$ of the test function $\Phi_{\ep, \vr_0}$ is different from Escobar's (the function $W$ in (4.2) of \cite{Es}),
but they share common characteristics such as decay properties as proved in the next lemma.
\begin{lemma}\label{lemma_V_dec}
There are constants $C, \eta_1, \eta_2 > 0$ depending only on $n$ and $\ga$ such that
\begin{equation}\label{eq_V_dec}
|V_{\ep, \vr_0}(x)| \le C \ep^{n-2\ga \over 2} \quad \text{for any } x \in X \setminus B^N(0,\vr_0)
\end{equation}
and
\begin{equation}\label{eq_V_dec_2}
|\nabla_{\bx} V_{\ep, \vr_0}(x)| \le C \vr_0^{-\eta_1} \ep^{n-2\ga+2\eta_2 \over 2}
\quad \text{and} \quad
|\pa_N V_{\ep, \vr_0}(x)| \le C \rho_0^{-\eta_1} \( \ep^{n-2\ga+2\eta_2 \over 2} + x_N^{2\ga-1} \ep^{n+2\ga \over 2} \)
\end{equation}
for $x = (\bx, x_N) \in X \cap \(B^N(0,2\vr_0) \setminus B^N(0,\vr_0)\)$.
Also we have $\nabla V_{\ep, \vr_0} = 0$ in $X \setminus B^N(0,2\vr_0)$.
\end{lemma}
\begin{proof}
We observe from \eqref{eq_W_dec} and \eqref{eq_V} that
\[|V_{\ep, \vr_0}(x)| \le C \vr_0^{n-2\ga} \left[\({\ep^{n-2\ga \over 2} \over \vr_0^{n-2\ga}}\)
+ \left|W_{\ep}(x) - \alpha_{n,\ga} {\ep^{n-2\ga \over 2} \over |x|^{n-2\ga}}\right|\right] \le C \(\ep^{n-2\ga \over 2} + {\ep^{n-2\ga+2\vt_2 \over 2} \over \vr_0^{\vt_2}}\)
\le C \ep^{n-2\ga \over 2}\]
for all $\vr_0 \le |x| \le 2\vr_0$ and some $\vt_2 \in (0,1)$, so \eqref{eq_V_dec} follows.
One can derive \eqref{eq_V_dec_2} by making the use of both \eqref{eq_W_dec}, \eqref{eq_W_dec_2} and \eqref{eq_W_dec_3}.
We leave the details to the reader.
\end{proof}

Now we assert the following proposition, which suffices to conclude that the fractional Yamabe problem is solvable in this case.
\begin{prop}\label{prop_lcf}
For $n > 2\ga$ and $\ga \in (0,1)$, let $(X^{n+1}, g^+)$ be a Poincar\'e-Einstein manifold with conformal infinity $(M^n, [\hh])$ such that \eqref{eq_eig} has the validity.
Assume also that $M$ is locally conformally flat.
If $(\ox, \bg)$ is not conformally diffeomorphic to the standard closed unit ball $\overline{\mb^N}$ and Conjecture \ref{conj_pos} holds, then
\[0 < \overline{I}_{\hh}^{\ga}[\Phi_{\ep, \vr_0}] < \Lambda^{\ga}(\ms^n, [g_c]).\]
\end{prop}
\begin{proof}
The proof is divided into 3 steps.

\medskip \noindent \textsc{Step 1: Estimation in $X \cap B^N(0,\vr_0)$.}
Applying \eqref{eq_bubble_eq}, \eqref{eq_Lambda}, \eqref{eq_bg_exp_2}, \eqref{eq_W_dec_2}, \eqref{eq_W_dec_3}, Lemma \ref{lemma_W_dec_2} and integrating by parts, we obtain
\begin{equation}\label{eq_exp_1}
\begin{aligned}
&\ \kappa_{\ga} \int_{X \cap B^N(0,\vr_0)} x_N^{1-2\ga} |\nabla W_{\ep}|_{\bg}^2 dv_{\bg} \\
&\le \Lambda^{\ga}(\ms^n, [g_c]) \(\int_{B^n(0,\vr_0)} w_{\ep,0}^{p+1} d\bx\)^{n-2\ga \over n} 
+ \kappa_{\ga} \int_{{X \cap \pa B^N(0,\vr_0)}} x_N^{1-2\ga} W_{\ep} {\pa W_{\ep} \over \pa \nu} \, dS \\
&\ + \underbrace{O\(\int_{B^n(0,\vr_0)} x_N^{n+1-2\ga} |\nabla W_{\ep}|^2 d\bx\)}_{= O\(\vr_0^{2\ga} \ep^{n-2\ga}\)}
\end{aligned} \end{equation}
where $\nu$ is the outward unit normal vector and $dS$ is the Euclidean surface measure.
On the other hand, if we write $g^+ = x_N^{-2} (dx_N^2 + h_{x_N})$, then
\begin{equation}\label{eq_E_1}
E(x_N) = - \({n-2\ga \over 4}\) x_N^{-2\ga} \text{tr}\,(h_{x_N}^{-1} \pa_N h_{x_N}) = O(x_N^{n-1-2\ga})
\end{equation}
in $X \cap B^N(0,2\vr_0)$ (see \eqref{eq_E_2}).
Therefore
\begin{equation}\label{eq_exp_2}
\kappa_{\ga} \int_{X \cap B^N(0,\vr_0)} E(x_N) W_{\ep}^2\, dv_{\bg} = O\(\vr_0^{2\ga} \ep^{n-2\ga}\).
\end{equation}

\medskip \noindent \textsc{Step 2: Estimation in $X \setminus B^N(0,\vr_0)$.}
By its own definition \eqref{eq_Phi} of the test function $\Phi_{\ep, \vr_0}$, its energy on $\ox$ can be evaluated as
\begin{align*}
&\ \int_{X \setminus B^N(0,\vr_0)} \(\rho^{1-2\ga} |\nabla \Phi_{\ep, \vr_0}|_{\bg}^2 + E(\rho) \Phi_{\ep, \vr_0}^2\) dv_{\bg}\\
&= \int_{X \setminus B^N(0,\vr_0)} \(\rho^{1-2\ga} \la \nabla (V_{\ep, \vr_0}^2 G), \nabla G \ra_{\bg}
+ E(\rho) V_{\ep, \vr_0}^2 G^2 + \rho^{1-2\ga} |\nabla V_{\ep, \vr_0}|^2 (G-\chi_{\vr_0}\Psi)^2\) dv_{\bg} \\
&\ + \int_{X \cap \(B^N(0,2\vr_0) \setminus B^N(0,\vr_0)\)} \rho^{1-2\ga}\({1 \over 2} \la \nabla V_{\ep, \vr_0}^2, \nabla(-2G \chi_{\vr_0}\Psi + \chi_{\vr_0}^2\Psi^2) \ra_{\bg}\) dv_{\bg} \\
&\ + \int_{X \cap \(B^N(0,2\vr_0) \setminus B^N(0,\vr_0)\)} \rho^{1-2\ga} V_{\ep, \vr_0}^2 \(|\nabla (\chi_{\vr_0}\Psi)|^2 - 2\la \nabla G, \nabla(\chi_{\vr_0}\Psi)\ra_{\bg}\) dv_{\bg} \\
&\ + \int_{X \cap \(B^N(0,2\vr_0) \setminus B^N(0,\vr_0)\)} E(\rho) V_{\ep, \vr_0}^2 \(\chi_{\vr_0}^2\Psi^2 - 2G\chi_{\vr_0}\Psi\) dv_{\bg}
\end{align*}
where $G = G(\cdot, 0)$.
From \eqref{eq_G}, \eqref{eq_Psi}, \eqref{eq_E_1} and Lemma \ref{lemma_V_dec}, we see that
\begin{equation}\label{eq_exp_3}
\begin{aligned}
&\ \kappa_{\ga} \int_{X \setminus B^N(0,\vr_0)} \(\rho^{1-2\ga} |\nabla \Phi_{\ep, \vr_0}|_{\bg}^2 + E(\rho) \Phi_{\ep, \vr_0}^2\) dv_{\bg}\\
&\le - \kappa_{\ga} \int_{{X \cap \pa B^N(0,\vr_0)}} x_N^{1-2\ga} V_{\ep, \vr_0}^2G\, {\pa G \over \pa \nu} (1+O(x_N^n)) \, dS
+ C \ep^{n-2\ga+2\eta_2} \vr_0^{-(n-2\ga-2+2\eta_1)}\\
&\ + C \ep^{n-2\ga + \eta_2} \vr_0^{\min\{\vt_1, 2\ga\}+1-\eta_1}
+ C \ep^{n-2\ga} \vr_0^{\min\{\vt_1, 2\ga\}}
\end{aligned}
\end{equation}
where $\vt_1 \in (0,1)$ and $C > 0$ depends only on $n,\, \ga,\, \vr_1$ and $\vr_2$. For instance, we have
\begin{align*}
&\ \int_{X \setminus B^N(0,\vr_0)} \rho^{1-2\ga} |\nabla V_{\ep, \vr_0}|^2 (G-\chi_{\vr_0}\Psi)^2 dv_{\bg} \\
&\le C \vr_0^{-2\eta_1} \int_{B^N(0,2\vr_0) \setminus B^N(0,\vr_0)} x_N^{1-2\ga} \(\ep^{n-2\ga+2\eta_2} + x_N^{2(2\ga-1)} \ep^{n+2\ga}\) \cdot \({1 \over |x|^{2(n-2\ga)}} + 1\) dx \\
&\le C \(\ep^{n-2\ga+2\eta_2} \vr_0^{-(n-2\ga-2+2\eta_1)} + \ep^{n+2\ga} \vr_0^{-n+6\ga} |\log \vr_0|\)
\le C \ep^{n-2\ga+2\eta_2} \vr_0^{-(n-2\ga-2+2\eta_1)}
\end{align*}
for $0 < \ep \ll \vr_0$ small. The other terms can be managed in a similar manner.

\medskip \noindent \textsc{Step 3: Conclusion.}
By combining \eqref{eq_exp_1}, \eqref{eq_exp_2} and \eqref{eq_exp_3}, we deduce
\begin{equation}\label{eq_exp_4}
\begin{aligned}
&\ \kappa_{\ga} \int_X \(\rho^{1-2\ga} |\nabla \Phi_{\ep, \vr_0}|_{\bg}^2 + E(\rho) \Phi_{\ep, \vr_0}^2\) dv_{\bg} \\
&\le \Lambda^{\ga}(\ms^n, [g_c]) \(\int_{B^n(0,\vr_0)} w_{\ep,0}^{p+1} d\bx\)^{n-2\ga \over n}
+ \kappa_{\ga} \int_{X \cap \pa B^N(0,\vr_0)} \underbrace{x_N^{1-2\ga} \(W_{\ep} {\pa W_{\ep} \over \pa \nu} - V_{\ep, \vr_0}^2 G\, {\pa G \over \pa \nu}\)}_{=: I} \, dS\\
&\ + C \ep^{n-2\ga} \vr_0^{\min\{\vt_1, 2\ga\}}.
\end{aligned}
\end{equation}
Let us compute the integral of $I$ over the boundary $X \cap \pa B^N(0,\vr_0)$ in the right-hand side of \eqref{eq_exp_4}.
Because of Lemma \ref{lemma_W_dec} and \eqref{eq_Psi}, one has
\begin{align*}
{\pa W_{\ep} \over \pa \nu} - V_{\ep, \vr_0}{\pa G \over \pa \nu}
&\le - {\alpha_{n,\ga}(n-2\ga) \ep^{n-2\ga \over 2} \over \vr_0^{n-2\ga+1}} + \(\vr_0^{-(n-2\ga)} + A\)^{-1} {\alpha_{n,\ga}(n-2\ga)\ep^{n-2\ga \over 2} \over \vr_0^{2(n-2\ga)+1}} \\
&\ + C \ep^{n-2\ga \over 2} \vr_0^{\min\{0,2\ga-1\}} + C \ep^{{n-2\ga \over 2}+\vt_2} \vr_0^{-(n-2\ga+1+\vt_2)} \\
&\le - \alpha_{n,\ga} (n-2\ga) A {\ep^{n-2\ga \over 2} \over \vr_0} + C \ep^{n-2\ga \over 2} \vr_0^{\min\{0,2\ga-1\}} + C \ep^{{n-2\ga \over 2}+\vt_2} \vr_0^{-(n-2\ga+1+\vt_2)}
\end{align*}
on $\{|x| = \vr_0\}$ for some $\vt_2 \in (0,1)$. Therefore using the fact that $W_1(x) \ge \alpha_{n,\ga}\ep^{n-2\ga \over 2}\vr_0^{-(n-2\ga)}/2$ on $\{|x| = \vr_0\}$, we discover
\begin{align*}
\int_{X \cap \pa B^N(0,\vr_0)} I\, dS
&= \int_{X \cap \pa B^N(0,\vr_0)} x_N^{1-2\ga} \left[W_{\ep} \({\pa W_{\ep} \over \pa \nu} - V_{\ep, \vr_0}{\pa G \over \pa \nu}\) - V_{\ep, \vr_0}^2 {\pa G \over \pa \nu} \Psi\right] dS\\
&\le - {\alpha_{n,\ga}^2 (n-2\ga) \over 4} \(\int_{\pa B^N(0,1)} |x_N|^{1-2\ga} dS\) A \ep^{n-2\ga} + C \ep^{n-2\ga} \vr_0^{\min\{1,2\ga\}} \\
&\ + C \ep^{n-2\ga+\vt_2} \vr_0^{-(n-2\ga+\vt_2)} + C \ep^{n-2\ga} \vr_0^{\vt_1+n}.
\end{align*}
Now the previous estimate, \eqref{eq_exp_4} and \eqref{eq_Lambda} yield that
\begin{align*}
\overline{I}_{\hh}^{\ga}[\Phi_{\ep, \vr_0}] &\le \Lambda^{\ga}(\ms^n, [g_c]) - {\alpha_{n,\ga}^2 \kappa_{\ga} (n-2\ga) \over 8S_{n,\ga}} \cdot {|\ms^{n-1}| \over 2} B\(1-\ga, {n \over 2}\) \cdot A \ep^{n-2\ga} \\
&\ + C \ep^{n-2\ga+\vt_2} \vr_0^{-(n-2\ga+\vt_2)} + C \ep^{n-2\ga} \vr_0^{\min\{\vt_1, 2\ga\}} \\
&< \Lambda^{\ga}(\ms^n, [g_c])
\end{align*}
where $B$ is the Beta function.
Additionally the last strict inequality holds for $0 < \ep \ll \vr_0$ small enough.
This completes the proof.
\end{proof}

\subsection{Two Dimensional Case} \label{subsec_2_dim}
We are now led to treat the case when $(M, [\hh])$ is a 2-dimensional closed manifold.

\medskip
Fix an arbitrary point $p \in M$ and let $\bx = (x_1, x_2)$ be normal coordinates at $p$.
Since $X$ is Poincar\'e-Einstein, it holds $h^{(1)} = 0$ in \eqref{eq_bg_exp}, whence we have
\begin{equation}\label{eq_bg_exp_3}
\bg_{ij}(\bx,x_N) = \delta_{ij} + O(|x|^2) \quad \text{and} \quad
|\bg| = 1 + O(|x|^2)
\quad \text{for } (\bx, x_N) \in \mcr^N(\vr_1, \vr_2) \subset \ox
\end{equation}
where the rectangle $\mcr^N(\vr_1, \vr_2)$ is defined in the line following \eqref{eq_bg_exp}.

\medskip
With Proposition \ref{prop_y_exist} in the introduction, the next result will give the validity of Theorem \ref{thm_main} if $n = 2$.
\begin{prop}
For $\ga \in (0,1)$, let $(X^3, g^+)$ be a Poincar\'e-Einstein manifold with conformal infinity $(M^2, [\hh])$ such that \eqref{eq_eig} holds.
If $(\ox, \bg)$ is not conformally diffeomorphic to the standard unit ball $\overline{\mb^3}$ and Conjecture \ref{conj_pos} holds, then
\[0 < \overline{I}_{\hh}^{\ga}[\Phi_{\ep, \vr_0}] < \Lambda^{\ga}(\ms^2, [g_c])\]
for the test function $\Phi_{\ep, \vr_0}$ introduced in \eqref{eq_Phi}.
\end{prop}
\begin{proof}
We compute the error in $X \cap B^N_+(0, \vr_0)$ due to the metric. As in \eqref{eq_exp_1} and \eqref{eq_exp_2}, one has
\[\int_{X \cap B^N(0,\vr_0)} x_N^{1-2\ga} |\nabla W_{\ep}|_{\bg}^2 dv_{\bg}
= \int_{X \cap B^N(0,\vr_0)} x_N^{1-2\ga} |\nabla W_{\ep}|^2 dx + \underbrace{O\(\int_{X \cap B^N(0,\vr_0)} x_N^{1-2\ga} |x|^2 |\nabla W_{\ep}|^2 dx\)}_{= O\(\vr_0^{2\ga} \ep^{2-2\ga}\)}\]
and
\[\int_{X \cap B^N(0,\vr_0)} E(x_N) W_{\ep}^2\, dv_{\bg} = O\(\int_{X \cap B^N(0,\vr_0)} x_N^{1-2\ga} W_{\ep}^2\, dx\) = O\(\vr_0^{2\ga} \ep^{2-2\ga}\)\]
from \eqref{eq_bg_exp_3}.
Therefore the error arising from the metric is ignorable, and the same argument in proof of Proposition \ref{prop_lcf} works.
The proof is completed.
\end{proof}

\medskip \noindent
{\footnotesize \textit{Acknowledgement.} S. Kim is supported by FONDECYT Grant 3140530.
M. Musso is partially supported by FONDECYT Grant 1120151 and Millennium Nucleus Center for Analysis of PDE, NC130017.
The research of J. Wei is partially supported by NSERC of Canada.
Part of the paper was finished while S. Kim was visiting the University of British Columbia and Wuhan University in 2015.
He appreciates the institutions for their hospitality and financial support.}

\appendix
\section{Expansion of the Standard Bubble $W_{1,0}$ near Infinity} \label{sec_app}
This appendix is devoted to find expansions of the function $W_1 = W_{1,0}$ (defined in \eqref{eq_bubble}) and its derivatives near infinity.
Especially we improve \cite[Lemma A.2]{CK} by pursuing a new approach based on conformal properties of $W_1$.

\medskip
For the functions $W_1$ and $x \cdot \nabla W_1$, we have
\begin{lemma}\label{lemma_W_dec}
Suppose that $n > 2\ga$ and $\ga \in (0,1)$. For any fixed large number $R_0 > 0$, we have
\begin{equation}\label{eq_W_dec}
\left|W_1(x) - {\alpha_{n,\ga} \over |x|^{n-2\ga}}\right| + \left|x \cdot \nabla W_1(x) + {\alpha_{n,\ga}(n-2\ga) \over |x|^{n-2\ga}}\right| \le {C \over |x|^{n-2\ga+\vt_2}}
\end{equation}
for $|x| \ge R_0$, where numbers $\vt_2 \in (0,1)$ and $C > 0$ rely only on $n,\, \ga$ and $R_0$.
\end{lemma}
\begin{proof}
Given any function $F$ in $\mr^N_+$, let $F^*$ be its fractional Kelvin transform defined as
\[F^*(x) = {1 \over |x|^{n-2\ga}}\, F\({x \over |x|^2}\) \quad \text{for } x \in \mr^N_+.\]
Then it is known that $W_1^* = W_1$.
Let us claim that $(x \cdot \nabla W_1)^*(0) = - \alpha_{n,\ga} (n-2\ga)$ and $(x \cdot \nabla W_1)^*$ is $C^{\infty}$ in the $\bx$-variable and H\"older continuous in the $x_N$-variable.
Since
\[x_N^{2-2\ga} \pa_{NN} W_1 = -(1-2\ga) x_N^{1-2\ga} \pa_N W_1 - x_N^{2-2\ga} \Delta_{\bx} W_1 \quad \text{in } \mr^N_+,\]
we have
\[\begin{cases}
-\text{div}\(x_N^{1-2\ga} \nabla (x \cdot \nabla W_1)\) = 0 &\text{in } \mr^N_+,\\
\begin{aligned}
\pa^{\ga}_{\nu} (x \cdot \nabla W_1) &= \sum\limits_{i=1}^n x_i \pa_{x_i} \pa^{\ga}_{\nu} W_1
+ \pa^{\ga}_{\nu} W_1 - \lim\limits_{x_N \to 0} x_N^{2-2\ga} \pa_{NN} W_1\\
&= p \sum\limits_{i=1}^n x_i \pa_{x_i}(w_1^p) + 2\ga w_1^p
\end{aligned}
&\text{on } \mr^n.
\end{cases}\]
Employing \cite[Proposition 2.6]{FW}, \cite{CS} and doing some computations, we obtain that
\[\begin{cases}
-\text{div}\(x_N^{1-2\ga} \nabla (x \cdot \nabla W_1)^*\) = 0 &\text{in } \mr^N_+,\\
\pa^{\ga}_{\nu} (x \cdot \nabla W_1)^* = (-\Delta)^{\ga} (x \cdot \nabla W_1)^* = \alpha_{n,\ga}^p \(\dfrac{2\ga|\bx|^2 - n}{(1+|\bx|^2)^{n+2\ga+2 \over 2}}\) &\text{on } \mr^n.
\end{cases}\]
Therefore $(x \cdot \nabla W_1)^*$ has regularity stated above, and according to Green's representation formula,
\[(x \cdot \nabla W_1)^*(0) = \alpha_{n,\ga}^p g_{n,\ga} \int_{\mr^n} {1 \over |\by|^{n-2\ga}} \(\dfrac{2\ga|\by|^2 - n}{(1+|\by|^2)^{n+2\ga+2 \over 2}}\)\, d\by = - \alpha_{n,\ga} (n-2\ga).\]
This proves the assertion.

Now we can check \eqref{eq_W_dec} with the above observations.
By standard elliptic theory, there exist constants $c_1, \cdots, c_N > 0$ such that
\begin{equation}\label{eq_W_dec_1}
\left| W_1^*(x) - \alpha_{n,\ga} \right| + \left| (x \cdot \nabla W_1)^*(x) + \alpha_{n,\ga}(n-2\ga)\right| \le \sum_{i=1}^n c_i |x_i| + c_N x_N^{\vt_2}
\end{equation}
for any $|x| \le R_0^{-1}$ and some $\vt_2 \in (0,1)$.
Hence, by taking the Kelvin transform in \eqref{eq_W_dec_1}, we see that the desired inequality \eqref{eq_W_dec} is valid for all $|x| \ge R_0$.
\end{proof}

Besides we have the following decay estimate of the derivatives of $W_1$.
\begin{lemma}\label{lemma_W_dec_2}
Assume that $n > 2\ga$ and $\ga \in (0,1)$.
For any fixed large number $R_0 > 0$, there exist constants $C > 0$ and $\vt_3 \in (0,\min\{1,2\ga\})$ depending only on $n,\, \ga$ and $R_0$ such that
\begin{equation}\label{eq_W_dec_2}
\left|\nabla_{\bx} W_1(x) + {\alpha_{n,\ga}(n-2\ga)\bx \over |x|^{n-2\ga+2}}\right| \le {C \over |x|^{n-2\ga+1+\vt_3}}
\end{equation}
and
\begin{equation}\label{eq_W_dec_3}
\left|\pa_N W_1(x) + {\alpha_{n,\ga}(n-2\ga)x_N \over |x|^{n-2\ga+2}}\right| \le C \( {1 \over |x|^{n-2\ga+2}} + {x_N^{2\ga-1} \over |x|^{n+2\ga}} \)
\end{equation}
for $|x| \ge R_0$.
\end{lemma}
\begin{proof}
The precise values of the constants $p_{n,\ga}$, $\alpha_{n,\ga}$ and $\kappa_{\ga}$, which will appear during the proof, are found in \eqref{eq_const}.

\medskip \noindent \textsc{Step 1.} By \eqref{eq_bubble}, \eqref{eq_bubble_2} and Taylor's theorem, it holds
\begin{align*}
\pa_i W_1(x) 
&= p_{n,\ga} \int_{\mr^n} {1 \over (|\by|^2+1)^{n+2\ga \over 2}}\, \pa_i w_1(\bx - x_N \by)\, d\by \\
&= p_{n,\ga} \int_{\mr^n} {1 \over (|\by|^2+1)^{n+2\ga \over 2}}\, \left[\pa_i w_1(- x_N \by) + \pa_{ij} w_1(- x_N \by) x_j + O(|\bx|^2) \right]\, d\by \\
&= p_{n,\ga} \int_{\mr^n} {1 \over (|\by|^2+1)^{n+2\ga \over 2}}\, \left[\pa_{ii} w_1(0) x_i + O((x_N |\by|)^{\vt_3}|\bx|) + O(|\bx|^2) \right]\, d\by \\
&= -\alpha_{n,\ga} (n-2\ga) x_i + O(|x|^{1+\vt_3})
\end{align*}
for $|x| \le R_0^{-1}$.
Here we also used the facts that the $C^2(\mr^n)$-norm of $w_1$ and the $C^{\vt_3}(\mr^n)$-norm of $\pa_{ij} w_1$ are bounded for some $\vt_3 \in (0,\min\{1,2\ga\})$.
On the other hand, the uniqueness of the $\ga$-harmonic extension yields that $(\pa_i W_1)^* = \pa_i W_1$ for $i = 1, \cdots, n$.
Therefore
\[\left| \pa_i W_1(x) + {\alpha_{n,\ga} (n-2\ga) x_i \over |x|^{n-2\ga+2}} \right| = \left| (\pa_i W_1)^*(x) + \alpha_{n,\ga} (n-2\ga) x_i^* \right| \le C (|x|^{1+\vt_3})^* \le {C \over |x|^{n+2\ga+1+\vt_3}} \]
for $|x| \ge R_0$, which is the desired inequality \eqref{eq_W_dec_2}.

\medskip \noindent \textsc{Step 2.}
If $\ga = 1/2$, it is known that
\[W_1(\bx, x_N) = \alpha_{n,1/2} \({1 \over |\bx|^2 + (x_n+1)^2}\)^{n-1 \over 2} \quad \text{for all } (\bx, x_N) \in \mr^N_+,\]
so direct computation shows
\[\left|\pa_N W_1(x) + {\alpha_{n,1/2}(n-1)x_N \over |x|^{n+1}}\right| \le {C \over |x|^{n+1}},\]
thereby implying \eqref{eq_W_dec_3}. Therefore it is sufficient to consider when $\ga \in (0,1) \setminus \{1/2\}$.
In light of duality \cite[Subsection 2.3]{CS}, we have that
\[\begin{cases}
- \text{div}\(x_N^{1-2(1-\ga)} \nabla \(x_N^{1-2\ga} \pa_N W_1\)\) = 0 &\text{in } \mr^N_+,\\
x_N^{1-2\ga} \pa_N W_1 = - \kappa_{\ga}^{-1} w_1^p &\text{on } \mr^n.
\end{cases}\]
Hence if we define
\[F^{**}(x) = {1 \over |x|^{n-2(1-\ga)}}\, F\({x \over |x|^2}\) \quad \text{for } x \in \mr^N_+.\]
for an arbitrary function $F$ in $\mr^N_+$, then
\[\begin{cases}
- \text{div}\(x_N^{1-2(1-\ga)} \nabla \(x_N^{1-2\ga} \pa_N W_1\)^{**}\) = 0 &\text{in } \mr^N_+,\\
\(x_N^{1-2\ga} \pa_N W_1\)^{**} = - \alpha_{n,\ga}^p \kappa_{\ga}^{-1} \dfrac{|\bx|^2}{(1 + |\bx|^2)^{n+2\ga \over 2}} &\text{on } \mr^n.
\end{cases}\]
This implies
\begin{equation}\label{eq_pa_N_W}
\begin{aligned}
\(x_N^{1-2\ga} \pa_N W_1\)^{**}(\bx, x_N)
&= - \alpha_{n,\ga}^p \kappa_{\ga}^{-1} p_{n,1-\ga}\, x_N^{2-2\ga}
\int_{\mr^n} {1 \over |\by|^{n-2\ga}} {1 \over (1 + |\by|^2)^{n+2\ga \over 2}}\, d\by + O\( x_N^{2-2\ga} |x| + |x|^2 \) \\
&= - \alpha_{n,\ga}(n-2\ga) x_N^{2-2\ga} + O\( x_N^{2-2\ga} |x| + |x|^2 \)
\end{aligned}
\end{equation}
for all $|x| \le R_0^{-1}$, where estimation of the remainder term is deferred to the end of the proof. Accordingly, we have
\[\left|x_N^{1-2\ga} \pa_N W_1(x) + {\alpha_{n,\ga}(n-2\ga) x_N^{2-2\ga} \over |x|^{n-2\ga+2}}\right| \le C \({x_N^{2-2\ga} \over |x|^{n-2\ga+3}} + {1 \over |x|^{n+2\ga}} \) \]
for $|x| \ge R_0$. Dividing the both sides by $x_N^{1-2\ga}$ finishes the proof of \eqref{eq_W_dec_3}.

\medskip \noindent \textbf{Estimation of the remainder term in \eqref{eq_pa_N_W}.}
The remainder term is equal to a constant multiple of
\begin{equation}\label{eq_rem_2}
\begin{aligned}
&\ \int_{\mr^n} \left[ {1 \over (1 + |\by|^2)^{n-2\ga+2 \over 2}} \cdot {(|\bx|^2 - 2x_N \bx \cdot \by + |x_N \by|^2) \over (1 + |x_N \by - \bx|^2)^{n+2\ga \over 2}}
- {1 \over |\by|^{n-2\ga+2}} \cdot {|x_N \by|^2 \over (1 + |x_N \by|^2)^{n+2\ga \over 2}} \right]\, d\by \\
&= O(|\bx|^2) + O\(x_N |\bx| \int_{\mr^n} {1 \over (1 + |\by|^2)^{n-2\ga+2 \over 2}} \cdot {|\by| \over (1 + |x_N \by - \bx|^2)^{n+2\ga \over 2}} d\by\) \\
& \quad + x_N^2 \int_{\mr^n} \left[{|\by|^2 \over (1 + |\by|^2)^{n-2\ga+2 \over 2}} \cdot {1 \over (1 + |x_N \by - \bx|^2)^{n+2\ga \over 2}}
- {1 \over |\by|^{n-2\ga}} \cdot {1 \over (1 + |x_N \by|^2)^{n+2\ga \over 2}} \right] d\by\\
& = \begin{cases}
O(|\bx|^2) + O(x_N |\bx|) + \left[O(x_N^2) + O(x_N^{2-2\ga} |\bx|)\right] &\text{for } \ga < 1/2,\\
O(|\bx|^2) + O(x_N^{2-2\ga}|\bx|) + \left[O(x_N^2) + O(x_N^{2-2\ga} |\bx|)\right] &\text{for } \ga > 1/2.
\end{cases}
\end{aligned}
\end{equation}
The estimate for the third term in the middle side of \eqref{eq_rem_2} can be done as
\begin{equation}\label{eq_rem_21}
\begin{aligned}
&\ x_N^2 \int_{\mr^n}  \left|{|\by|^2 \over (1 + |\by|^2)^{n-2\ga+2 \over 2}} - {1 \over |\by|^{n-2\ga}} \right| \cdot {1 \over (1 + |x_N \by - \bx|^2)^{n+2\ga \over 2}} d\by \\
&= O(x_N^2) + O\(x_N^{4-2\ga} \int_{\{|\by| > x_N\}} {1 \over |\by|^{n-2\ga+2}} \cdot {1 \over (1 + |\by-\bx|^2)^{n+2\ga \over 2}}  d\by\)
= O(x_N^2)
\end{aligned}
\end{equation}
with the aid of Taylor's theorem and the substitution $x_N \by \to \by$, and
\begin{align*}
&\ x_N^2 \int_{\mr^n} {1 \over |\by|^{n-2\ga}} \cdot \left[ {1 \over (1 + |x_N \by - \bx|^2)^{n+2\ga \over 2}} - {1 \over (1 + |x_N \by|^2)^{n+2\ga \over 2}} \right] d\by \\
&=  x_N^{2-2\ga} \int_{\mr^n} {1 \over |\by|^{n-2\ga}} \cdot \left[ {1 \over (1 + |\by - \bx|^2)^{n+2\ga \over 2}} - {1 \over (1 + |\by|^2)^{n+2\ga \over 2}} \right] d\by = O(x_N^{2-2\ga} |\bx|).
\end{align*}
Also we estimated the second term in the middle side of \eqref{eq_rem_2} by decomposing $\mr^n$ into two regions $\{|\by| \le 1\}$ and $\{|\by| > 1\}$ as in \eqref{eq_rem_21}.
This concludes the proof.
\end{proof}

\section{Some Integrations Regarding the Standard Bubble $W_{1,0}$ on $\mr^N_+$} \label{sec_app_2}
The following lemmas are due to Gonz\'alez-Qing \cite[Section 7]{GQ} and the authors \cite[Subsection 4.3]{KMW}.
\begin{lemma}\label{lemma_KMW_1}
Suppose that $n > 4\ga - 1$.
For each $x_N > 0$ fixed, let $\whw_1(\xi,x_N)$ be the Fourier transform of $W_1(\bx,x_N)$ with respect to the variable $\bx \in \mr^n$.
In addition, we use $K_{\ga}$ to signify the modified Bessel function of the second kind of order $\ga$.
Then we have that
\[\whw_1(\xi,x_N) = \hw_1(\xi)\, \vp(|\xi|x_N) \quad \text{for all } \xi \in \mr^n \text{ and } x_N > 0,\]
where $\vp(t) = d_1 t^{\ga}K_{\ga}(t)$ is the solution to
\begin{equation}\label{eq_phi_1}
\phi''(t) + {1-2\ga \over t} \phi'(t) - \phi(t) = 0, \quad \phi(0) = 1 \text{ and } \phi(\infty) = 0
\end{equation}
and $\hw_1(t) := \hw_1(|\xi|) = d_2 |\xi|^{-\ga}K_{\ga}(|\xi|)$ solves
\begin{equation}\label{eq_phi_2}
\phi''(t) + {1+2\ga \over t} \phi'(t) - \phi(t) = 0 \quad \text{and} \quad  \lim_{t \to 0} t^{2\ga}\phi(t) + \lim_{t \to \infty} t^{\ga+{1 \over 2}} e^t \phi(t) \le C
\end{equation}
for some $C > 0$. The numbers $d_1,\, d_2 > 0$ depend only on $n$ and $\ga$.
\end{lemma}

\begin{lemma}\label{lemma_KMW_2}
Let
\begin{equation}\label{eq_AB}
\begin{aligned}
&A_{\alpha} = \int_0^{\infty} t^{\alpha-2\ga}\vp^2(t)\, dt, &B_{\alpha} = \int_0^{\infty} t^{-\alpha+2\ga} \hw_1^2(t) t^{n-1} dt,\\
&A_{\alpha}' = \int_0^{\infty} t^{\alpha-2\ga} \vp(t)\, \vp'(t)\, dt, &B_{\alpha}' = \int_0^{\infty} t^{-\alpha+2\ga} \hw_1(t)\, \hw_1'(t) t^{n-1} dt,\\
&A_{\alpha}'' = \int_0^{\infty} t^{\alpha-2\ga} (\vp'(t))^2\, dt, &B_{\alpha}'' = \int_0^{\infty} t^{-\alpha+2\ga} (\hw_1'(t))^2 t^{n-1} dt
\end{aligned}
\end{equation}
for ${\alpha} \in \mn \cup \{0\}$. Then
\begin{align*}
A_{\alpha} &= \({\alpha+2 \over \alpha+1}\) \cdot \left[\({\alpha+1 \over 2}\)^2 - \ga^2\right]^{-1} A_{\alpha+2} = -\({\alpha+1 \over 2} - \ga\)^{-1} A'_{\alpha+1} \\
&= \({\alpha+1 \over 2} - \ga\) \({\alpha-1 \over 2} + \ga\)^{-1} A''_{\alpha}
\end{align*}
for $\alpha$ odd, $\alpha \ge 1$ and
\[B_{\alpha} = {4(n-\alpha+1)B_{\alpha-2} \over (n-\alpha)(n+2\ga-\alpha)(n-2\ga-\alpha)} = - {2 B'_{\alpha-1} \over n+2\ga-\alpha},
\, B_{\alpha-2} = {(n-2\ga-\alpha)  B''_{\alpha-2} \over n+2\ga-\alpha+2}\]
for $\alpha$ even, $\alpha \ge 2$.
\end{lemma}
\begin{proof}
Apply \eqref{eq_phi_1}, \eqref{eq_phi_2} and the identity
\[\int_0^{\infty} t^{\alpha-1} u(t)\, u'(t) dt = -\({\alpha-1 \over 2}\) \int_0^{\infty} t^{\alpha-2} u(t)^2 dt\]
which holds for any $\alpha > 1$ and $u \in C^1(\mr)$ decaying sufficiently fast.
\end{proof}

Utilizing the above lemmas, we compute some integrals regarding the standard bubble $W_1$ and its derivatives.
The next identities are necessary in the energy expansion when non-minimal conformal infinities are considered. See Subsection \ref{subsec_nm_ene}.
\begin{lemma}\label{lemma_int_0}
Suppose that $n \ge 2$ and $\ga \in (0,1/2)$. Then
\[\int_{\mr^N_+} x_N^{2-2\ga} |\nabla W_1|^2 dx = \({4 \over 1+2\ga}\) \int_{\mr^N_+} x_N^{2-2\ga} (\pa_r W_1)^2 dx = \({1-2\ga \over 2}\) \int_{\mr^N_+} x_N^{-2\ga} W_1^2 dx < \infty.\]
\end{lemma}
\begin{proof}
Refer to \cite[Lemma 6.3]{CK}.
\end{proof}

The following information is used in the energy expansion for the non-umbilic case. Refer to Subsections \ref{subsec_num_ene} and \ref{subsec_num_ene_2}.
\begin{lemma}\label{lemma_int}
For $n > 2 + 2\ga$, it holds that 
\begin{align*}
\mcf_1 &:= \int_{\mr^N_+} x_N^{1-2\ga} W_1^2 dx = \left[{3 \over 2\(1-\ga^2\)}\right] |\ms^{n-1}| A_3B_2,\\
\mcf_2 &:= \int_{\mr^N_+} x_N^{3-2\ga} |\nabla W_1|^2 dx = \({3 \over 1+\ga}\) |\ms^{n-1}| A_3B_2,\\
\mcf_3 &:= \int_{\mr^N_+} x_N^{3-2\ga} (\pa_r W_1)^2 dx = |\ms^{n-1}| A_3B_2,\\
\mcf_4 &:= \int_{\mr^N_+} x_N^{3-2\ga} r(\pa_r W_1)(\pa_{rr} W_1) dx = -{n \over 2} |\ms^{n-1}| A_3B_2,\\
\mcf_5 &:= \int_{\mr^N_+} x_N^{3-2\ga} r^2 (\pa_{rr} W_1)^2 dx = \left[{5n^3 - 4n(1+\ga^2) + 4(1-4\ga^2) \over 20(n-1)}\right] |\ms^{n-1}| A_3B_2,\\
\mcf_6 &:= \int_{\mr^N_+} x_N^{1-2\ga} r^2 (\pa_r W_1)^2 dx = \left[{(n+2) (3n^2-6n+4-4\ga^2) \over 8(n-1)(1-\ga^2)}\right] |\ms^{n-1}| A_3B_2,\\
\mcf_7 &:= \int_{\mr^N_+} x_N^{2-2\ga} r^2 (\pa_r W_1)(\pa_{rx_N} W_1) dx = -\left[{(n+2) (3n^2-6n+4-4\ga^2) \over 8(n-1)(1+\ga)}\right] |\ms^{n-1}| A_3B_2,\\
\mcf_8 &:= \int_{\mr^N_+} x_N^{3-2\ga} r^2 (\pa_{rx_N} W_1)^2 dx = \left[{(2-\ga) (5n^3 - 4n(2-2\ga+\ga^2) + 8(1-\ga-2\ga^2)) \over 20(n-1)(1+\ga)}\right] |\ms^{n-1}| A_3B_2.
\end{align*}
Here $r = |\bx|$, and the positive constants $A_3$ and $B_2$ are defined by \eqref{eq_AB}.
\end{lemma}
\begin{proof}
The values $\mcf_1$, $\mcf_2$, $\mcf_3$ and $\mcf_6$ were computed in \cite{GQ, KMW}, so it suffices to consider the others.

\medskip \noindent \textsc{Step 1 (Calculation of $\mcf_4$).} Integration by parts gives
\begin{align*}
\mcf_4 &= \int_{\mr^N_+} x_N^{3-2\ga} r(\pa_r W_1)(\pa_{rr} W_1) dx = |\ms^{n-1}| \int_0^{\infty} x_N^{3-2\ga} \({1 \over 2} \int_0^{\infty} r^n \pa_r(\pa_r W_1)^2 dr\) dx_N\\
&= |\ms^{n-1}| \int_0^{\infty} x_N^{3-2\ga} \(-{n \over 2} \int_0^{\infty} r^{n-1} (\pa_r W_1)^2 dr\) dx_N = -{n \over 2} \mcf_3
= -{n \over 2} |\ms^{n-1}| A_3B_2.
\end{align*} 

\medskip \noindent \textsc{Step 2 (Calculation of $\mcf_5$).}
Since $\Delta_{\bx} W_1 = W_1'' + (n-1)r^{-1}W_1'$ (where $'$ stands for the differentiation in $r$), it holds that
\begin{equation}\label{eq_F_1}
\int_{\mr^N_+} x_N^{3-2\ga} r^2 (\Delta_{\bx} W_1)^2 dx = \mcf_5 + 2(n-1) \mcf_4 + (n-1)^2 \mcf_3.
\end{equation}
By the Plancherel theorem, Lemma \ref{lemma_KMW_1} and the relation
\begin{multline*}
\Delta_{\xi} (|\xi|^2 \hw_1(|\xi|) \vp(|\xi|x_N)) = 2n \hw_1 \vp + (n+2-2\ga) |\xi| \hw_1' \vp + (n+2+2\ga) |\xi| \hw_1 \vp' x_N\\
+ |\xi|^2 \hw_1 \vp + 2|\xi|^2 \hw_1' \vp' x_N + |\xi|^2 \hw_1 \vp x_N^2
\end{multline*}
where the variable of $\hw_1$ and $\hw_1'$ is $|\xi|$, that of $\vp$ and $\vp'$ is $|\xi|x_N$, and $'$ represents the differentiation with respect to the radial variable $|\xi|$, we see
\begin{align*}
&\ \int_{\mr^N_+} x_N^{3-2\ga} r^2 (\Delta_{\bx} W_1)^2 dx \\
&= \int_0^{\infty} x_N^{3-2\ga} \int_{\mr^n} (-\Delta_{\xi}) (|\xi|^2 \hw_1(|\xi|) \vp(|\xi|x_N)) \cdot (|\xi|^2 \hw_1(|\xi|) \vp(|\xi|x_N)) \, d\xi \, dx_N \\
&= |\ms^{n-1}| \left[2n A_3B_2 + (n+2-2\ga) A_3B_1' + (n+2+2\ga) A_4'B_2 + A_3B_0 + 2A_4'B_1' + A_5B_2 \right].
\end{align*}
Therefore Lemma \ref{lemma_KMW_2} implies
\[\int_{\mr^N_+} x_N^{3-2\ga} r^2 (\Delta_{\bx} W_1)^2 dx = \left[{5n^3 - 20n^2 + 4n(9-\ga^2) - 16(1+\ga^2) \over 20(n-1)}\right] |\ms^{n-1}| A_3B_2.\]
Now \eqref{eq_F_1} and the information on $\mcf_3$ and $\mcf_4$ yield the desired estimate for $\mcf_5$.

\medskip \noindent \textsc{Step 3 (Calculation of $\mcf_7$ and $\mcf_8$).}
Since the basic strategy is similar to Step 2, we will just sketch the proof. We observe
\begin{align*}
\mcf_7 &= {1 \over 2} \int_0^{\infty} x_N^{2-2\ga} \pa_N \(\int_{\mr^n} r^2 (\pa_r W_1)^2 d\bx\) dx_N
= {1 \over 2} \int_0^{\infty} x_N^{2-2\ga} \pa_N \(\sum_{i=1}^n \int_{\mr^n} |\bx|^2 (\pa_{x_i} W_1)^2 d\bx\) dx_N \\
&= {1 \over 2} \int_0^{\infty} x_N^{2-2\ga} \underbrace{\pa_N \(\sum_{i=1}^n \int_{\mr^n} (-\Delta_\xi) (\xi_i \hw_1(|\xi|) \vp(|\xi|x_N))
\cdot (\xi_i \hw_1(|\xi|) \vp(|\xi|x_N)) d\xi\)}_{= (I)} dx_N.
\end{align*}
Owing to Lemmas \ref{lemma_KMW_1} and \ref{lemma_KMW_2}, one can compute the term
\begin{align*}
(I) &= - \left[(n+1) \int_{\mr^n} \pa_N \(|\xi| (\hw_1 \hw_1')(|\xi|)\, \vp^2(|\xi|x_N) + |\xi| \hw_1^2(|\xi|)\, (\vp \vp')(|\xi|x_N) x_N\)\, d\xi \right.\\
&\qquad + \int_{\mr^n} \pa_N \(|\xi|^2 (\hw_1 \hw_1'')(|\xi|)\, \vp^2(|\xi|x_N) + 2|\xi|^2 (\hw_1 \hw_1')(|\xi|)\, (\vp \vp')(|\xi|x_N) x_N\)\, d\xi\\
&\qquad \left. + \int_{\mr^n} \pa_N \(|\xi|^2 \hw_1^2(|\xi|)\, (\vp \vp'')(|\xi|x_N) x_N^2\)\, d\xi \right]
\end{align*}
to get the value of $\mcf_7$ in the statement of the lemma. Moreover,
\begin{align*}
\mcf_8 &= \int_0^{\infty} x_N^{3-2\ga} \(\int_{\mr^n} |\bx|^2 |\nabla_{\bx} (\pa_N W_1)|^2 d\bx\) dx_N \\
&= \int_0^{\infty} x_N^{3-2\ga} \(\sum_{i=1}^n \int_{\mr^n} (-\Delta_{\xi})(\xi_i \pa_N \whw_1) \cdot (\xi_i \pa_N \whw_1) d\xi\) dx_N.
\end{align*}
The rightmost term is computable with Lemmas \ref{lemma_KMW_1} and \ref{lemma_KMW_2}.
The proof is completed.
\end{proof}

The next lemma lists the values of some integrals which are needed in the energy expansion for the umbilic case (see Subsection \ref{subsec_um_ene}).
Its proof is analogous to the proofs of Lemma \ref{lemma_int} and \cite[Lemma 4.4]{KMW}, so we skip it.
\begin{lemma}\label{lemma_int_2}
For $n > 4 + 2\ga$, we have
\begin{align*}
\mcf'_1 &:= \int_{\mr^N_+} x_N^{3-2\ga} W_1^2 dx = \left[{4(n-3) \over (n-4)(n-4-2\ga)(n-4+2\ga)}\right] |\ms^{n-1}| A_3B_2,\\
\mcf'_2 &:= \int_{\mr^N_+} x_N^{5-2\ga} |\nabla W_1|^2 dx = \left[{16(n-3)(2-\ga) \over (n-4)(n-4-2\ga)(n-4+2\ga)}\right] |\ms^{n-1}| A_3B_2,\\
\mcf'_3 &:= \int_{\mr^N_+} x_N^{5-2\ga} (\pa_r W_1)^2 dx = \left[{16(n-3)(4-\ga^2) \over 5(n-4)(n-4-2\ga)(n-4+2\ga)}\right] |\ms^{n-1}| A_3B_2,\\
\mcf'_4 &:= \int_{\mr^N_+} x_N^{1-2\ga} r^2 W_1^2 dx = \left[{n(3n^2 - 18n + 28 - 4\ga^2) \over 2(n-4)(n-4-2\ga)(n-4+2\ga)(1-\ga^2)}\right] |\ms^{n-1}| A_3B_2,\\
\mcf'_5 &:= \int_{\mr^N_+} x_N^{3-2\ga} r^2 |\nabla W_1|^2 dx = \left[{n(3n^2 + 2n(-7+2\ga) - 4(-4+3\ga+\ga^2)) \over (n-4)(n-4-2\ga)(n-4+2\ga)(1+\ga)}\right] |\ms^{n-1}| A_3B_2,\\
\mcf'_6 &:= \int_{\mr^N_+} x_N^{3-2\ga} r^2 (\pa_r W_1)^2 dx = \left[{(n+2)(5n^2 - 20n + 16 - 4\ga^2) \over 5(n-4)(n-4-2\ga)(n-4+2\ga)}\right] |\ms^{n-1}| A_3B_2,\\
\mcf'_7 &:= \int_{\mr^N_+} x_N^{5-2\ga} r(\pa_r W_1)(\pa_{rr} W_1) dx = -\left[{8n(n-3)(4-\ga^2) \over 5(n-4)(n-4-2\ga)(n-4+2\ga)}\right] |\ms^{n-1}| A_3B_2,\\ 
\mcf'_8 &:= \int_{\mr^N_+} x_N^{5-2\ga} r^2 (\pa_{rr} W_1)^2 dx = \left[{4(4-\ga^2)(7n^3 - 14n^2 - 4n(5 + \ga^2) + 4 - 16\ga^2) \over 35(n-4)(n-4-2\ga)(n-4+2\ga)}\right] |\ms^{n-1}| A_3B_2,\\
\mcf'_9 &:= \int_{\mr^N_+} x_N^{4-2\ga} r^2 (\pa_r W_1)(\pa_{rx_N} W_1) dx = -\left[{(n+2)(2-\ga)(5n^2 - 20n + 16 - 4\ga^2) \over 5(n-4)(n-4-2\ga)(n-4+2\ga)}\right] |\ms^{n-1}| A_3B_2,\\
\mcf'_{10} &:= \int_{\mr^N_+} x_N^{5-2\ga} r^2 (\pa_{rx_N} W_1)^2 dx = \left[\tfrac{4(2-\ga)(3-\ga) (7n^3 - 14n^2 - 4n(6 - 2\ga + \ga^2) + 8(2 - 3\ga - 2\ga^2))} {35(n-4)(n-4-2\ga)(n-4+2\ga)}\right] |\ms^{n-1}| A_3B_2
\end{align*}
where $r = |\bx|$, and the positive constants $A_3$ and $B_2$ are defined by \eqref{eq_AB}.
\end{lemma}

\Addresses

{\footnotesize
\begin{thebibliography}{1}
\bibitem{Al}
S. Almaraz, \emph{An existence theorem of conformal scalar-flat metrics on manifolds with boundary}, Pacific J. Math. \textbf{248} (2010), 1--22.

\bibitem{Au}
T. Aubin, \emph{\'{E}quations diff\'{e}rentielles non lin\'{e}aires et probl\`{e}me de Yamabe concernant la courbure scalaire}, J. Math. Pures Appl. \textbf{55} (1976), 269--296.

\bibitem{B}
T. P. Branson, \emph{Differential operators canonically associated to a conformal structure}, Math. Scand. \textbf{57} (1985), 293--345.

\bibitem{Br}
S. Brendle, \emph{Blow-up phenomena for the Yamabe equation}, J. Amer. Math. Soc. \textbf{21} (2008), 951--979.

\bibitem{BC}
S. Brendle and S. Chen, \emph{An existence theorem for the Yamabe problem on manifolds with boundary}, J. Eur. Math. Soc. (JEMS) \textbf{16} (2014), 991--1016.

\bibitem{BS}
H. Brezis and W. A. Strauss, \emph{Semi-linear second-order elliptic equations in $L^1$}, J. Math. Soc. Japan \textbf{25} (1973), 565--590.

\bibitem{CS}
L.~Caffarelli and L.~Silvestre, \emph{An extension problem related to the fractional {L}aplacian}, Comm. Partial Differential Equations \textbf{32} (2007), 1245--1260.

\bibitem{Ca}
J. S. Case, \emph{Some energy inequalities involving fractional GJMS operators}, preprint, arXiv:1509.08347.

\bibitem{CC}
J. S. Case and S.-Y. A. Chang, \emph{On fractional GJMS operators}, preprint, to appear in Comm. Pure Appl. Math., arXiv:1406.1846.

\bibitem{CG}
S.-Y.~A. Chang and M.~d.~M. Gonz\'{a}lez, \emph{Fractional {L}aplacian in conformal geometry}, Adv. Math. \textbf{226} (2011), 1410--1432.

\bibitem{CY}
S.-Y. A. Chang and P. C. Yang, \emph{Extremal metrics of zeta functional determinants on 4-manifolds}, Ann. of Math. \textbf{142} (1995), 171--212.

\bibitem{CK}
W. Choi and S. Kim, \emph{On perturbations of the fractional Yamabe problem}, preprint, arXiv:1501.00641.

\bibitem{DM}
Z. Djadli and A. Malchiodi, \emph{Existence of conformal metrics with constant Q-curvature.} Ann. of Math. \textbf{168} (2008), 813--858.

\bibitem{Es}
J.~F. Escobar, \emph{Conformal deformation of a Riemannian metric to a scalar flat metric with constant mean curvature on the boundary}, Ann. of Math. \textbf{136} (1992), 1--50.

\bibitem{Es2}
\bysame, \emph{The Yamabe problem on manifolds with boundary}, J. Differential Geom. \textbf{35} (1992), 21--84.

\bibitem{Es3}
\bysame, \emph{Conformal metrics with prescribed mean curvature on the boundary}, Calc. Var. Partial Differential Equations \textbf{4} (1996), 559--592.

\bibitem{FKS}
E. Fabes, C. Kenig and R. Serapioni, \emph{The local regularity of solutions of degenerate elliptic equations}, Comm. Partial Differential Equations \textbf{7} (1982), 77--116.

\bibitem{FW}
M. M. Fall and T. Weth, \emph{Nonexistence results for a class of fractional elliptic boundary value problems}, J. Funct. Anal. \textbf{263} (2012), 2205--2227.

\bibitem{GQ}
M.~d.~M. Gonz\'{a}lez and J.~Qing, \emph{Fractional conformal {L}aplacians and fractional {Y}amabe problems}, Analysis and PDE \textbf{6} (2013), 1535--1576.

\bibitem{GW}
M.~d.~M. Gonz\'{a}lez and M. Wang, \emph{Further results on the fractional Yamabe problem: the umbilic case}, preprint, arXiv:1503.02862.

\bibitem{Gr}
C. R. Graham, \emph{Volume and area renormalizations for conformally compact Einstein metrics},
The Proceedings of the 19th Winter School ``Geometry and Physics" (Srn, 1999). Rend. Circ. Mat. Palermo (2) Suppl. \textbf{63} (2000), 31--42.

\bibitem{GZ}
C. R. Graham and M. Zworski, \emph{Scattering matrix in conformal geometry}, Invent. Math. \textbf{152} (2003), 89--118.

\bibitem{GQ2}
C. Guillarmou and J. Qing, \emph{Spectral characterization of Poincar\'e-Einstein manifolds with infinity of positive Yamabe type}, Int. Math. Res. Not. \textbf{2010} (2010), 1720--1740.

\bibitem{GO}
P. Gurka and B. Opic, \emph{Continuous and compact imbeddings of weighted Sobolev spaces. II}, Czechoslovak Math. J. \textbf{39} (1989), 78--94.

\bibitem{Gu}
M. J. Gursky, \emph{The principal eigenvalue of a conformally invariant differential operator, with an application to semilinear elliptic PDE}, Comm. Math. Phys. \textbf{207} (1999), 131--143.

\bibitem{GM}
M. J. Gursky and A. Malchiodi, \emph{A strong maximum principle for the Paneitz operator and a non-local flow for the Q -curvature}, J. Eur. Math. Soc. (JEMS) \textbf{17} (2015), 2137--2173.

\bibitem{HY}
F. Hang and P. C. Yang, \emph{The Sobolev inequality for Paneitz operator on three manifolds}, Calc. Var. Partial Differential Equations \textbf{21} (2004), 57--83.

\bibitem{HY2}
\bysame, \emph{Sign of Green’s function of Paneitz operators and the Q curvature}, Int. Math. Res. Not. \textbf{2015} (2015), 9775--9791.

\bibitem{HY3}
\bysame, \emph{Q curvature on a class of 3 manifolds}, to appear in Comm. Pure Appl. Math.

\bibitem{HY4}
\bysame, \emph{Q-Curvature on a class of manifolds with dimension at least 5}, to appear in Comm. Pure Appl. Math.

\bibitem{HR}
E. Humbert and S. Raulot, \emph{Positive mass theorem for the Paneitz-Branson operator}, Calc. Var. Partial Differential Equations \textbf{36} (2009), 525--531.

\bibitem{JX}
T. Jin and J. Xiong, \emph{Sharp constants in weighted trace inequalities on Riemannian manifolds}, Calc. Var. Partial Differential Equations \textbf{48} (2013), 555--585.

\bibitem{KMW}
S. Kim, M. Musso and J. Wei, \emph{A non-compactness result on the fractional Yamabe problem in large dimensions}, preprint, arXiv:1505.06183.

\bibitem{KMW2}
\bysame, \emph{A compactness theorem of the fractional Yamabe problem on the non-umbilic conformal infinity}, work in progress.

\bibitem{LP}
J. M. Lee and T. H. Parker, \emph{The Yamabe problem}, Bull. Amer. Math. Soc. \textbf{17} (1987), 37--91.

\bibitem{Ma}
F. Marques, \emph{Existence results for the Yamabe problem on manifolds with boundary}, Indiana Univ. Math. J. \textbf{54} (2005), 1599--1620.

\bibitem{Ma2}
\bysame, \emph{Conformal deformations to scalar-flat metris with constant mean curvature on the boundary}, Comm. Anal. Geom. \textbf{15} (2007), 381--405.

\bibitem{Pa}
S. M. Paneitz, \emph{A quartic conformally covariant differential operator for arbitrary pseudo-Riemannian manifolds (summary)}, SIGMA Symmetry Integrability Geom. Methods Appl. \textbf{4} (2008), Paper 036, 3 pp.

\bibitem{QR}
J. Qing and D. Raske, \emph{On positive solutions to semilinear conformally invariant equations on locally conformally flat manifolds}, Int. Math. Res. Not. \textbf{2006} (2006), Article ID 94172, 1--20.

\bibitem{Sc}
R. Schoen, \emph{Conformal deformation of a Riemannian metric to constant scalar curvature}, J. Differential Geom. \textbf{20} (1984), 479--495.

\bibitem{SY1}
R. Schoen and S.-T. Yau, \emph{On the proof of the positive mass conjecture in general relativity}, Commun. Math. Phys. \textbf{65} (1979), 45--76. 

\bibitem{SY2}
\bysame, \emph{Proof of the positive action conjecture in quantum relativity}, Phys. Rev. Lett. \textbf{42} (1979) 547--548. 

\bibitem{SY4}
\bysame, \emph{Conformally flat manifolds, Kleinian groups, and scalar curvature}, Invent. Math. \textbf{92} (1988), 47--71. 

\bibitem{Tr}
N. Trudinger, \emph{Remarks concerning the conformal deformation of Riemannian structures on compact manifolds}, Ann. Scuola Norm. Sup. Pisa Cl. Sci. \textbf{22} (1968), 265--274.

\bibitem{Xi}
J. Xiao, \emph{A sharp Sobolev trace inequality for the fractional-order derivatives}, Bull. Sci. Math. \textbf{130} (2006), 87--96.

\bibitem{Ya}
H. Yamabe, \emph{On a deformation of Riemannian structures on compact manifolds}, Osaka Math. J. \textbf{12} (1960), 21--37.
\end{thebibliography}}
\end{document}